\documentclass[10pt, reqno]{amsart}

\usepackage{amsmath,amssymb,amsthm,amsfonts,verbatim}
\usepackage{microtype}
\usepackage[all,2cell]{xy}
\usepackage{mathtools}
\usepackage{graphicx}
\usepackage{pinlabel}
\usepackage{hyperref}
\usepackage{mathrsfs}
\usepackage{color}
\usepackage{enumerate}
\usepackage{cite}
\usepackage{subcaption}

\CompileMatrices

\usepackage[top=1.2in,bottom=1.8in,left=1.2in,right=1.2in]{geometry}

\theoremstyle{plain}
\newtheorem{theorem}{Theorem}[section]
\newtheorem{maintheorem}{Theorem}

\newtheorem{proposition}[theorem]{Proposition}
\newtheorem{lemma}[theorem]{Lemma}

\newtheorem{corollary}[theorem]{Corollary}

\theoremstyle{definition}
\newtheorem{definition}[theorem]{Definition}
\newtheorem{example}[theorem]{Example}
\newtheorem{construction}[theorem]{Construction}

\newtheorem{remark}[theorem]{Remark}

\newcommand{\nc}{\newcommand}
\nc{\dmo}{\DeclareMathOperator}

\nc{\Q}{\mathbb{Q}}
\nc{\F}{\mathbb{F}}
\nc{\R}{\mathbb{R}}
\nc{\Z}{\mathbb{Z}}
\nc{\C}{\mathbb{C}}
\nc{\Ell}{\mathcal{L}}
\nc{\M}{\mathcal{M}}
\nc{\K}{\mathcal{K}}
\nc{\I}{\mathcal{I}}
\nc{\T}{\mathcal T}
\nc{\U}{\mathcal U}
\nc{\disk}{\mathbb{D}}
\nc{\hyp}{\mathbb{H}}

\nc{\CP}{\mathbb{CP}}
\nc{\cS}{\mathcal{S}}
\dmo{\Mod}{Mod}
\dmo{\PMod}{PMod}
\dmo{\LMod}{LMod}
\dmo{\Diff}{Diff}
\dmo{\Homeo}{Homeo}
\dmo{\dist}{dist}
\dmo\BDiff{BDiff}
\dmo\SO{SO}
\dmo\Hom{Hom}
\dmo\SL{SL}
\dmo\Sp{Sp}
\dmo\rank{rank}
\dmo\sig{sig}
\dmo\Out{Out}
\dmo\Aut{Aut}
\dmo\Inn{Inn}
\dmo\GL{GL}
\dmo\PSL{PSL}
\dmo\BHomeo{BHomeo}
\dmo\EHomeo{EHomeo}
\dmo\EDiff{EDiff}
\nc\Sig{\Sigma}
\dmo\Teich{Teich}
\dmo\Fix{Fix}
\nc{\pair}[1]{\langle #1 \rangle}
\nc{\abs}[1]{\left| #1 \right|}
\nc{\action}{\circlearrowright}
\nc{\norm}[1]{\left | \left | #1 \right | \right |}
\nc{\abcd}[4]{\left(\begin{array}{cc} #1 & #2 \\ #3 & #4 \end{array}\right)}
\nc{\into}\hookrightarrow
\dmo{\Isom}{Isom}
\nc{\normal}{\vartriangleleft}
\dmo{\Vol}{Vol}
\dmo{\im}{Im}
\dmo{\Push}{Push}
\dmo{\Conf}{Conf}
\dmo{\PConf}{PConf}
\dmo{\id}{id}
\dmo{\Jac}{Jac}
\dmo{\Pic}{Pic}
\dmo{\Stab}{Stab}
\dmo{\Arf}{Arf}
\dmo{\End}{End}
\dmo{\Gal}{Gal}
\dmo{\lcm}{lcm}
\dmo{\ab}{ab}
\dmo{\opp}{op}
\dmo{\SU}{SU}
\dmo{\HT}{\Omega \mathcal{T}}
\dmo{\HM}{\Omega \mathcal{M}}
\dmo{\spin}{spin}
\dmo{\even}{even}
\dmo{\odd}{odd}

\nc{\Span}[1]{\operatorname{Span}(#1)}

\newcommand{\sing}{\underline{\kappa}}
\renewcommand{\epsilon}{\varepsilon}
\renewcommand{\tilde}{\widetilde}
\renewcommand{\le}{\leqslant}
\nc{\coloneq}{\mathrel{\mathop:}\mkern-1.2mu=}
\nc{\margin}[1]{\marginpar{\scriptsize #1}}
\nc{\para}[1]{\medskip\noindent\textbf{#1.}}
\nc{\red}[1]{\textcolor{red}{#1}}
\nc{\blue}[1]{\textcolor{blue}{#1}}

\nc{\lb}{[}
\nc{\rb}{]}

\title[Higher spin mapping class groups and strata of Abelian differentials]{Higher spin mapping class groups and strata of Abelian differentials over Teichm{\"u}ller space}

\author{Aaron Calderon and Nick Salter}
\email{aaron.calderon@yale.edu}
\email{nks@math.columbia.edu}
\thanks{AC is supported by NSF Award No. DGE-1122492. NS is supported by NSF Awards DMS-1703181 and DMS-2003984.}

\address{AC: Department of Mathematics, Yale University, 10 Hillhouse Ave, New Haven, CT 06511}
\address{NS: Department of Mathematics, Columbia University, 2990 Broadway, New York, NY 10027}
\date{June 28, 2021}

\begin{document}
\begin{abstract}
For $g \ge 5$, we give a complete classification of the connected components of strata of abelian differentials over Teichm\"uller space, establishing an analogue of Kontsevich and Zorich's classification of their components over moduli space. Building on work of the first author \cite{Calderon_strata}, we find that the non-hyperelliptic components are classified by an invariant known as an $r$--spin structure. This is accomplished by computing a certain monodromy group valued in the mapping class group. To do this, we determine explicit finite generating sets for all $r$--spin stabilizer subgroups of the mapping class group, completing a project begun by the second author in \cite{Salter_monodromy}. Some corollaries in flat geometry and toric geometry are obtained from these results. 
\end{abstract}
\vspace*{-2.5em}
\maketitle
\vspace*{-1em}

\section{Introduction}\label{section:introduction}

The moduli space $\HM_g$ of abelian differentials is a vector bundle over the usual moduli space $\M_g$ of closed genus $g$ Riemann surfaces, whose fiber above a given $X \in \M_g$ is the space $\Omega(X)$ of abelian differentials (holomorphic 1--forms) on $X$. Similarly, the space $\HT_g$ of marked abelian differentials is a vector bundle over the Teichm{\"u}ller space $\T_g$ of {\em marked} Riemann surfaces of fixed genus (recall that a {\em marking} of $X \in \M_g$ is an isotopy class of map from a (topological) reference surface $\Sigma_g$ to $X$).

Both $\HM_g$ and $\HT_g$ are naturally partitioned into subspaces called {\em strata} by the number and order of the zeros of a differential appearing in the stratum. For a partition $\sing = (k_1, \ldots, k_n)$ of $2g-2$, define
\[
\HM(\sing) := \{ (X, \omega) \in \HM_g :
\omega \in \Omega(X) \text{ with zeros of order } k_1, \ldots, k_n\}.
\]
Define $\HT(\sing)$ similarly; then $\HM(\sing)$ is the quotient of $\HT(\sing)$ by the mapping class group $\Mod(\Sigma_g)$. Each stratum $\HM(\sing)$ is an orbifold, and the mapping class group action demonstrates $\HT(\sing)$ as an orbifold covering space of $\HM(\sing)$.

While strata are fundamental objects in the study of Riemann surfaces, their global structure is poorly understood (outside of certain special cases \cite{LM_strata}).
Kontsevich and Zorich famously proved that there are only ever at most three connected components of $\HM(\sing)$, depending on hyperellipticity and the ``Arf invariant'' of an associated spin structure (see Theorem \ref{theorem:KZstrata} and Definition \ref{definition:arfQF}).

In \cite{Calderon_strata}, the first author gives a partial classification of the non-hyperelliptic connected components of $\HT(\sing)$ in terms of invariants known as ``$r$--spin structures,'' certain $\Z/r\Z$-valued functions on the set of isotopy classes of oriented simple closed curves (c.f. Definition \ref{definition:spin}). Our first main theorem finishes that classification, settling Conjecture 1.3 of \cite{Calderon_strata} for all $g \ge 5$.

\begin{maintheorem}[Classification of strata]\label{theorem:classification}
Let $g \ge 5$ and $\sing = (k_1, \ldots, k_n)$ be a partition of $2g-2$. Set $r = \gcd(\sing)$. Then there are exactly $r^{2g}$ non--hyperelliptic components of $\HT(\sing)$, corresponding to the $r$--spin structures on $\Sigma_g$.

Moreover, when $\gcd(\sing)$ is even, exactly
$(r/2)^{2g} \left( 2^{g-1} (2^g + 1) \right)$ of these components have even Arf invariant and
$(r/2)^{2g} \left( 2^{g-1} (2^g - 1) \right)$ have odd Arf invariant. 
\end{maintheorem}

The classification of Theorem \ref{theorem:classification} should be contrasted with the classification of hyperelliptic components: as shown in the proof of \cite[Corollary 2.6]{Calderon_strata}, there are infinitely many hyperelliptic connected components for $g \ge 3$, corresponding to hyperelliptic involutions of the surface. 

We emphasize that Theorem \ref{theorem:classification} together with \cite[Corollary 2.6]{Calderon_strata} yields a complete classification of the connected components of $\HT(\sing)$ for all $g \ge 5$.\\

As in \cite{Calderon_strata}, the proof of Theorem \ref{theorem:classification} follows by analyzing which mapping classes can be realized as flat deformations living in a (connected component of a) stratum of genus-$g$ differentials. The {\em geometric monodromy group} of such a component $\mathcal H$ is the subgroup $\mathcal{G}({\mathcal H}) \le \Mod(\Sigma_g)$ obtained from the forgetful map $\mathcal H \to \mathcal M_g$ by taking (orbifold) $\pi_1$ (see also Remark \ref{rem:geomonodromy}). Theorem \ref{theorem:classification} is obtained as a direct corollary of the monodromy calculation described in Theorem \ref{theorem:moncomp} below. To formulate the result, let $\phi$ denote the $r$-spin structure associated to the stratum-component $\mathcal H$ (see Construction \ref{example:diffspin} for details), and let $\Mod_g[\phi]$ denote the stabilizer of $\phi$ under the action of $\Mod(\Sigma_g)$ (c.f. Definition \ref{definition:stabilizer}).

\begin{maintheorem}[Monodromy of strata]\label{theorem:moncomp}
Let $g \ge 5$ and $\sing = (k_1, \ldots, k_n)$ be a partition of $2g-2$. Let $\mathcal H$ be a non-hyperelliptic component of $\HT(\sing)$ with associated $r$-spin structure $\phi$. Then the geometric monodromy group $\mathcal{G}({\mathcal H}) \le \Mod(\Sigma_g)$ is the stabilizer of $\phi$:
\[
\mathcal{G}({\mathcal H}) = \Mod_g[\phi].
\]
\end{maintheorem}

The first author proved special cases of Theorems \ref{theorem:classification} and \ref{theorem:moncomp} in the case where $r \notin \{2g-2, g-1\}$, and when $r$ is even, only gave bounds on the number of components (equivalently, the index of $\mathcal{G}({\mathcal H})$ in $\Mod(\Sigma_g)$). Furthermore, those results only applied to a restricted family of genera.
The main difference in our work is that our improvements to the theory of stabilizers of $r$-spin structures yield exact computations for all genera $g \ge 5$ and all $r$.

The containment $\mathcal{G}({\mathcal H}) \le \Mod_g[\phi]$ is relatively clear; see Lemma \ref{lemma:rspinconstant}. The core of Theorem \ref{theorem:moncomp} is to show that $\mathcal{G}({\mathcal H})$ {\em surjects} onto $\Mod_g[\phi]$. Towards this goal, our final main theorem provides explicit finite generating sets for the stabilizer of any $r$--spin structure. In \cite[Theorem 9.5]{Salter_monodromy}, the second author obtained partial results in this direction, but the results there only applied in the setting of $r < g-1$, and were only approximate in the case of $r$ even. 

To state our results, we recall that the set of $r$--spin structures on $\Sigma_g$ is empty unless $r$ divides $2g-2$ (see Remark \ref{remark:rdivideschi}). For any $r$--spin structure $\phi$ on a surface of genus $g$, define
a {\em lift} of $\phi$ to be any $(2g-2)$--spin structure $\tilde \phi$ such that
\[\tilde \phi (c) \equiv \phi(c) \pmod r\]
for every oriented simple closed curve $c$.

Since $\Mod(\Sigma_g)$ acts transitively on the set of $r$--spin structures with the same Arf invariant (Lemma \ref{lemma:stabindex}), it suffices to give generators for the stabilizer of a single $r$--spin structure with given Arf invariant. There are two values of the Arf invariant, and hence two configurations of curves to be constructed in each genus. We analyze two particular such families of configurations in cases \ref{case:hardeven} and \ref{case:hardodd} of Theorem \ref{theorem:genset}. As a function of genus, the Arf invariant exhibits mod-$4$ periodicity in these families, leading to the dichotomy between $g \equiv 0,1 \pmod 4$ and $g \equiv 2,3 \pmod 4$ seen below.

\begin{maintheorem}[Generating $\Mod_g \lb \phi \rb$]\label{theorem:genset}
$\,$
\begin{enumerate} 
\item \label{case:hardeven} Let $g \ge 5$ be given. Then the Dehn twists about the curves $a_0, a_1, \dots, a_{2g-1}$ shown in Figure \ref{figure:hardcaseeven} generate $\Mod_g[\tilde \phi]$, where $\tilde \phi$ is the unique $(2g-2)$--spin structure obtained by specifying that $\tilde \phi(a_i) = 0$ for all $i$, and
\[
\Arf(\tilde \phi)  = \begin{cases} 
	1	&{g \equiv 0,3 \pmod 4}\\
	0	&{g \equiv 1,2 \pmod 4}
\end{cases}
\] 
\begin{figure}[ht]
\labellist
\small
\pinlabel $a_0$ [r] at 109.6 16
\pinlabel $a_1$ [b] at 88 91.2
\pinlabel $a_2$ [bl] at 110.4 59.2
\pinlabel $a_3$ [br] at 25.6 46.4
\pinlabel $a_4$ [t] at 77.6 28
\pinlabel $a_5$ [bl] at 124.2 45
\pinlabel $a_6$ [b] at 148 40
\pinlabel $a_{2g-1}$ [b] at 315.2 48.8
\endlabellist
\includegraphics{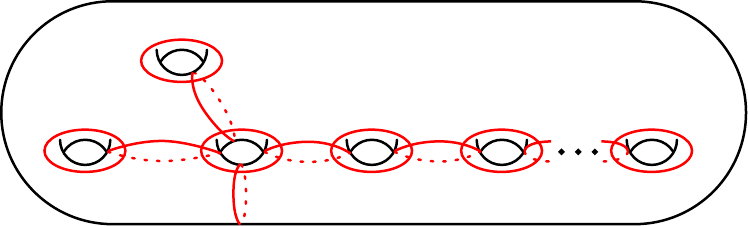}
\caption{Generators for $\Mod_g[\tilde \phi]$, Case \ref{case:hardeven}}
\label{figure:hardcaseeven}
\end{figure}

\item \label{case:hardodd} Let $g \ge 5$ be given. Then the Dehn twists about the curves $a_0, a_1, \dots, a_{2g-1}$ shown in Figure \ref{figure:hardcaseodd} generate $\Mod_g[\tilde \phi]$, where $\tilde \phi$ is the unique $(2g-2)$--spin structure obtained by specifying that $\tilde \phi(a_i) = 0$ for all $i$, and
\[
\Arf(\tilde \phi)  = \begin{cases} 
	1	&{g \equiv 1,2 \pmod 4}\\
	0	&{g \equiv 0,3 \pmod 4}
\end{cases}
\] 
\begin{figure}[ht]
\labellist
\small
\pinlabel $a_0$ [r] at 167.2 83.2
\pinlabel $a_1$ [b] at 52 66.4
\pinlabel $a_2$ [b] at 81.6 60.8
\pinlabel $a_{2g-1}$ [b] at 308.8 65.6
\endlabellist
\includegraphics{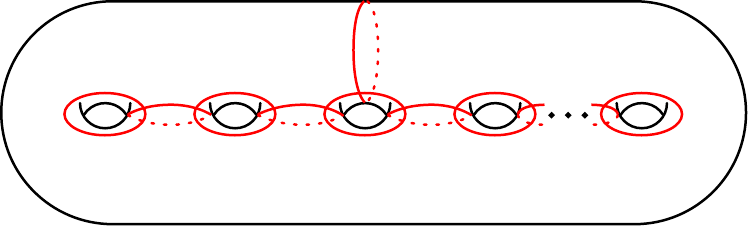}
\caption{Generators for $\Mod_g[\tilde \phi]$, Case \ref{case:hardodd}}
\label{figure:hardcaseodd}
\end{figure}

\item \label{case:r} Let $g \ge 3$, let $r$ be a proper divisor of $2g-2$, and let $\phi$ be an $r$--spin structure. Let $\tilde \phi$ be an arbitrary lift of $\phi$ to a $(2g-2)$--spin structure, and let $\{c_i\}$ be any collection of simple closed curves such that the set of values $\{\tilde \phi(c_i)\}$ generates the subgroup $r\Z/(2g-2)\Z$. Then 
\[
\Mod_g[\phi] = \pair{\Mod_g[\tilde \phi], \{T_{c_i}\}}.
\]
In particular, $\Mod_g[\phi]$ is generated by a finite collection of Dehn twists for all $g \ge 5$: the twists about the curves $\{c_i\}$ in combination with the finite generating set for $\Mod_g[\tilde \phi]$ given by whichever of Theorem \ref{theorem:genset}.\ref{case:hardeven} or \ref{theorem:genset}.\ref{case:hardodd} is applicable to $\tilde \phi$. 
\end{enumerate}
\end{maintheorem}

We remark that the generating sets exhibited in Cases \ref{case:hardeven} and \ref{case:hardodd} are minimal in the sense that any subcollection of the $a_i$'s does not generate $\Mod_g[\tilde \phi]$. Indeed, any proper subset does not fill $\Sigma_g$, and so the subgroup generated by those twists stabilizes a curve and is hence of infinite index in $\Mod(\Sigma_g)$.

\para{Application: realizing curves as cylinders} Using the monodromy calculation of Theorem \ref{theorem:moncomp}, we give a complete characterization of which curves can be realized as embedded Euclidean cylinders on some surface in a given connected component of a stratum (see Section \ref{subsection:flat} for a discussion of cylinders and other basic notions in flat geometry). 

\begin{corollary}\label{corollary:cylchar}
Suppose that $g \ge 5$ and $\sing$ is a partition of $2g-2$ with $\gcd(\sing) = r$. Let $\tilde{\mathcal H}$ be a component of $\HT(\sing)$ and $\phi$ the corresponding $r$--spin structure, and let $c \subset \Sigma_g$ be a simple closed curve.
\begin{itemize}
\item If $\tilde{ \mathcal H}$ is hyperelliptic with corresponding involution $\iota$, then $c$ is realized as the core curve of a cylinder on some marked abelian differential in $\tilde{ \mathcal H}$ if and only if it is nonseparating and $\iota(c)=c$.
\item If $\tilde { \mathcal H}$ is non-hyperelliptic, then $c$ is realized as the core curve of a cylinder on some marked abelian differential in $\tilde { \mathcal H}$ if and only if it is nonseparating and $\phi(c)=0$.
\end{itemize}
\end{corollary}

\para{Application: homological monodromy} As a second corollary, we recover a recent theorem of Guti\'errez--Romo \cite[Corollary 1.2]{GR_RVgroups} using topological methods. The result of Guti\'errez--Romo concerns the {\em homological} monodromy of a stratum. Let ${ \mathcal H}$ be a component of $\HM(\sing)$. There is a vector bundle $H_1{ \mathcal H}$ over ${ \mathcal H}$ where the fiber over the Abelian differential $(X, \omega)$ is the space $H_1(X, \R)$. The (orbifold) fundamental group of ${ \mathcal H}$ admits a monodromy action on $H_1{ \mathcal H}$ as a subgroup of $\Sp(2g,\Z)$; this was computed (via the ``Rauzy--Veech group'' of ${ \mathcal H}$) by Guti\'errez--Romo.

\begin{corollary}[c.f. Corollary 1.2 of \cite{GR_RVgroups}]\label{corollary:GR}
Suppose that $\sing = (k_1, \dots, k_n)$ is a partition of $2g-2$ such that $g \ge 5$, and set $r = \gcd(k_1, \dots, k_n)$. Let ${ \mathcal H}$ be a connected component of $\HM(\sing)$.
\begin{enumerate}
\item If $r$ is odd, then the monodromy group of $H_1{ \mathcal H}$ is the entire symplectic group $\Sp(2g,\Z)$.
\item If $r$ is even, then the monodromy group of $H_1 { \mathcal H}$ is the stabilizer in $\Sp(2g,\Z)$ of a mod-$2$ quadratic form $q$ associated to the spin structure on the chosen basepoint; in particular, it is a finite-index subgroup of $\Sp(2g, \Z)$ (see Section \ref{subsection:arf}). 
\end{enumerate}
\end{corollary}

The proof of Corollary \ref{corollary:GR} follows from Theorem \ref{theorem:classification} together with a description of the action of $r$-spin mapping class groups on homology; see the end of Section \ref{subsection:firststep} for details.

\para{Application: monodromy of line bundles on toric surfaces} In the course of proving Theorem \ref{theorem:genset}, we establish in Proposition \ref{prop:admissfull} that the group $\T_{\phi}$ of ``admissible twists'' (c.f. Definition \ref{definition:admissible}) generated by the Dehn twists in all nonseparating curves $c$ with $\phi(c) = 0$ is equal to the spin stabilizer subgroup $\Mod_g[\phi]$. Together with the main theorem of \cite{Salter_monodromy}, this is enough to settle a conjecture of the second author. We briefly describe the problem, referring the interested reader to \cite{Salter_monodromy} for a fuller discussion.

 Suppose that $\Ell$ is an ample line bundle on a smooth toric surface $Y$ for which the generic fiber $\Sigma_{g(\Ell)}$ has genus at least 5 and is not hyperelliptic. Let $| \Ell |$ denote the complete linear system of $\Ell$ and $\M( \Ell ) \subset | \Ell |$ the complement of the discriminant locus; then $\mathcal{M}(\Ell)$ supports a tautological family of Riemann surfaces. Let $\pi: \mathcal{E}(\Ell) \rightarrow \M(\Ell)$ be the corresponding $\Sigma_{g(\Ell)}$ bundle, and let
\[
\Gamma_{\mathcal{L}} \le \Mod(\Sigma_{g(\Ell)})
\]
denote the image of the monodromy representation of $\pi$. The paper \cite{Salter_monodromy} undertakes a study of such $\Gamma_{\mathcal L}$. In this context, $r$-spin structures arise algebro-geometrically as maximal roots of the the adjoint line bundle $\Ell \otimes K_Y$. It follows that $\Gamma_\mathcal{L}$ preserves the associated $r$-spin structure: $\Gamma_{\mathcal L} \le \Mod_g[\phi]$. In the case of $r$ odd, \cite{Salter_monodromy} shows that this is an equality, but for $r$ even, we were only able to show that the index $[\Mod(\Sigma_{g(\Ell)}): \Gamma_{\mathcal{L}}]$ is finite. The improvements in the theory of $r$-spin mapping class groups afforded by Proposition \ref{prop:admissfull} allows us to upgrade this containment to equality in all cases.

\begin{corollary}[c.f. Conjecture 1.4 of \cite{Salter_monodromy}]
Fix $Y, \Ell, \Gamma_{\mathcal L}, \phi$ as above. Then
\[\Gamma_\mathcal{L} = \Mod_g[\phi].\]
\end{corollary}
\begin{proof}
By \cite[Theorem A]{Salter_monodromy}, $\T_{\phi} \le \Gamma_{\Ell} \le \Mod_g[\phi]$. By Proposition \ref{prop:admissfull}, $\T_\phi = \Mod_g[\phi]$.
\end{proof}

\para{Relation to previous work..} The present paper should be viewed as a joint sequel to the works \cite{Calderon_strata} and \cite{Salter_monodromy}. So as to avoid a large amount of redundancy, we have aimed to give an exposition that is self--contained but does not dwell on background. The reader looking for a more thorough discussion of flat geometry is referred to \cite{Calderon_strata}, and the reader looking for more on $r$--spin structures is referred to \cite{Salter_monodromy}. We have also omitted the proofs of many statements that are essentially contained in our previous work. In some cases we require slight modifications of our results that cannot be cited directly; in this case, we have attempted to indicate the necessary modifications without repeating the arguments in their entirety. 

For the most part, the technology of \cite{Calderon_strata} does not need to be improved, and much of the content of Section \ref{section:classification} is included solely for the convenience of the reader. On the other hand, Theorem \ref{theorem:genset} is a substantial improvement over its counterpart \cite[Theorem 9.5]{Salter_monodromy}. The basic outline is the same, but many of the constituent arguments have been sharpened and simplified. The reader who is primarily interested in the theory of the stabilizer groups $\Mod_g[\phi]$ is encouraged to treat Theorem \ref{theorem:genset} as the ``canonical'' version, and is referred to \cite[Theorem 9.5]{Salter_monodromy} only as necessary.

It is worthwhile to situate our work on $r$-spin mapping class groups within the larger context of the literature. To our knowledge, $r$-spin mapping class groups were first investigated by Sipe in the papers \cite{sipe1, sipe2}; Sipe acknowledges inspiration from Mumford. Sipe works out formulas for the action of $\Mod(\Sigma_g)$ on the set of $r$-spin structures and obtains some fundamental results on the structure of the simultaneous stabilizer of {\em all} $r$-spin structures. Later, Randal-Williams \cite{RW} investigated the homological stability properties of families of $r$-spin mapping class groups on surfaces of increasing genus; as part of this work, he obtains a classification of $\Mod(\Sigma_g)$-orbits on the set of $r$-spin structures. We exploit this classification (recorded here as Lemma \ref{lemma:stabindex}) throughout the paper. In unpublished work \cite{kawazumi}, Kawazumi carried out a similar analysis under different conventions for isotopy along boundary components; as the surfaces we consider in this paper are closed, we do not use Kawazumi's work directly. Finally, the paper \cite{Salter_monodromy} by the second author begins the project of finding explicit finite generating sets for $r$-spin mapping class groups that is completed here as Theorem \ref{theorem:genset}.

\para{...and to subsequent work} Since this paper was first released, the authors have completed a sequel \cite{strata3}. In this work, we consider a refined version of the monodromy representation valued in the {\em punctured} mapping class group $\Mod(\Sigma_{g,n})$ --- the marked points record the locations of the zeroes of the differential. The fundamental invariant of the present paper, the $r$-spin structure, is then refined into a {\em framing} of the punctured surface $\Sigma_{g,n}$. We develop the counterpart to Theorem \ref{theorem:genset} (in fact, we prove a stronger, ``coordinate-free'' version), finding finite generating sets for these ``framed mapping class groups.'' We use this to prove refined versions of Theorems \ref{theorem:classification} and \ref{theorem:moncomp}.

It is worth stressing that the results of \cite{strata3} logically depend on the results established here. Our study of framed mapping class groups proceeds by induction on the number of punctures $n$, and the base case $n = 1$ hinges on the results established in this paper in the case $r = 2g-2$.

\para{Outline of Theorem \ref{theorem:genset}} The proof of Theorem \ref{theorem:genset} largely follows the outline of the proof of \cite[Theorem 9.5]{Salter_monodromy}, with one modification that allows for a cleaner argument with less casework. The result of \cite[Theorem 9.5]{Salter_monodromy} did not treat the maximal case $r = 2g-2$, but here we are able to do so. In fact, we find that the case of general $r$ described in Theorem \ref{theorem:genset}.\ref{case:r} follows very quickly from the maximal case (see Section \ref{section:generalr}). Accordingly, the bulk of the proof only treats the case $r = 2g-2$. 

The argument in the case $r = 2g-2$ proceeds in two stages. The first stage, presented in Section \ref{section:finitegen} as Proposition \ref{lemma:gammaadmiss}, shows that the finite collections of twists given in Theorem \ref{theorem:genset}.\ref{case:hardeven}/\ref{case:hardodd} generate an intermediate subgroup $\mathcal T_\phi \le \Mod_g[\phi]$ called the ``admissible subgroup'' (see Definition \ref{definition:admissible}). This is the group generated by Dehn twists about ``admissible curves'' (again see Definition \ref{definition:admissible}). 

The set of admissible curves determine a subgraph of the curve graph, and Proposition \ref{lemma:gammaadmiss} is proved by working one's way out in this complex, using combinations of admissible twists to ``acquire'' twists about curves further and further out in the complex. This is encapsulated in Lemma \ref{lemma:pushmakesT} (note that the connection with curve complexes is all contained within the proof of Lemma \ref{lemma:pushmakesT}, which is imported directly from \cite{Salter_monodromy}). The corresponding arguments in \cite{Salter_monodromy} made use of the existence of a certain configuration of curves which does not exist when $r \ge g-1$. Here we avoid this issue by directly showing that the configurations of Theorem \ref{theorem:genset} have the requisite properties (c.f. Lemmas \ref{lemma:hardeven}, \ref{lemma:hardodd}). 

The second step is to show that the admissible subgroup $\T_\phi$ coincides with the stabilizer $\Mod_g[\phi]$, an {\em a priori} larger group. This result appears as Proposition \ref{prop:admissfull}; the proof takes place in Section \ref{section:admiss=stab}. The method here is to show that both $\T_\phi$ and $\Mod_g[\phi]$ have the same intersection with the Johnson filtration on $\Mod(\Sigma_g)$ (c.f. Section \ref{subsection:firststep}). The outline exactly mirrors its counterpart in \cite{Salter_monodromy}: Proposition \ref{prop:admissfull} follows by assembling the three Lemmas \ref{step1}, \ref{step2}, \ref{step3}, each of which shows that $\T_\phi$ and $\Mod_g[\phi]$ behave identically with respect to a certain piece of the Johnson filtration. The arguments provided here are both sharper and in many cases simpler than their predecessors in \cite{Salter_monodromy}. In particular, the previous version of Lemma \ref{step2} required an intricate lower bound on genus which we replace here with the uniform (and optimal) requirement $g \ge 3$. The other main result of this section, Lemma \ref{step3}, also improves on its predecessor. The previous version of Lemma \ref{step3} was only applicable for $r$ odd, but here we are able to treat arbitrary $r$. The internal workings of this step have also been improved and are now substantially less coordinate--dependent. 

Prior to the work carried out in Sections \ref{section:finitegen} and \ref{section:admiss=stab}, in Section \ref{section:sliding} we prove a lemma we call the ``sliding principle'' (Lemma \ref{lemma:sliding}). This is a flexible tool for carrying out computations involving the action of Dehn twists on curves, and largely subsumes the work done in \cite[Appendix A]{Calderon_strata}. We believe that the sliding principle will be widely applicable to the study of the mapping class group. 

\begin{remark}
Theorem \ref{theorem:genset} requires $g \ge 5$. This is necessary in only one place in the argument, Lemma \ref{lemma:pushmakesT}. This lemma, which was proved in \cite{Salter_monodromy}, rests on the connectivity of a certain simplicial complex which is disconnected for $g < 5$. It is likely that Lemma \ref{lemma:pushmakesT} holds for $g \ge 3$, but to the best of the authors' knowledge, some substantial new ideas are needed to improve the range. Among other things, this would complete the classification of components of strata of marked abelian differentials in genera $3$ and $4$. 
\end{remark}

\para{Outline of Theorems \ref{theorem:classification} and \ref{theorem:moncomp}} The proofs of Theorems \ref{theorem:classification} and \ref{theorem:moncomp} in turn essentially follow those of their counterparts \cite[Theorems 1.1 and 1.2]{Calderon_strata}. To prove Theorem \ref{theorem:moncomp}, we construct in Section \ref{subsection:prototype} a square--tiled surface in each component of each stratum that has a set of cylinders in correspondence with the Dehn twist generators described in Theorem \ref{theorem:genset}. Each such cylinder gives rise to a Dehn twist in $\mathcal G({ \mathcal H})$ (Lemma \ref{lemma:cyltwist}), so Theorem \ref{theorem:genset} implies that this collection of Dehn twists causes $\mathcal G({ \mathcal H})$ to be ``as large as possible,'' leading to the monodromy computation of Theorem \ref{theorem:moncomp}. Theorem \ref{theorem:classification} then follows as a corollary via the basic theory of covering spaces and the orbit-stabilizer theorem (Proposition \ref{prop:monstab}).

\para{Acknowledgements} The authors would like to thank Paul Apisa for suggesting that Theorem \ref{theorem:classification} should lead to Corollary \ref{corollary:cylchar}. They would also like to acknowledge Ursula Hamenst\"adt and Curt McMullen for comments on a preliminary draft, as well as multiple anonymous referees whose close reading and helpful suggestions greatly contributed to the presentation of this paper. Finally, they thank Dick Hain for his interest in the project and a productive correspondence.

\setcounter{tocdepth}{1}
\tableofcontents

\section{Higher spin structures}
Theorem \ref{theorem:classification} asserts that the non--hyperelliptic components of strata of marked abelian differentials are classified by an object known as an ``$r$--spin structure.'' Here we introduce the basic theory of such objects. After defining spin structures and their stabilizer subgroups in Section \ref{subsection:spinbasics}, we explain how $r$--spin structures arise from vector fields in Section \ref{subsection:wnf}. In Section \ref{subsection:arf}, we connect the theory of $r$--spin structures to the classical theory of spin structures and quadratic forms on vector spaces in characteristic $2$.
	
\para{Reference convention} To streamline pointers to \cite{Salter_monodromy}, in this section, we adopt the convention of referring to \cite[Statement X.Y]{Salter_monodromy} as ``(SX.Y)''.

\subsection{Basic properties}\label{subsection:spinbasics}
	There are several points of view on $r$--spin structures: they can be defined algebro--geometrically as a root of the canonical bundle, topologically as a cohomology class, or as an invariant of isotopy classes of simple closed curves on surfaces. For a more complete discussion, including proofs of the claims below, see \cite[Section 3]{Salter_monodromy}. In this work we only need to study $r$--spin structures from the point of view of surface topology; this approach is originally due to Humphries and Johnson \cite{HJ_windingnumber} and has its roots in the earlier work \cite{Johnson_spin} of Johnson.

	\begin{definition}[$r$--spin structure (S3.1)]\label{definition:spin}
	Let $\Sigma_g$ be a closed surface of genus $g \ge 2$, and let $\mathcal S$ denote the set of isotopy classes of oriented simple closed curves on $\Sigma_g$; we include here the inessential curve $\zeta$ that bounds an embedded disk to its left. An {\em $r$--spin structure} is a function $\phi: \mathcal S \to \Z/r\Z$ satisfying the following two properties:
	\begin{enumerate}
	\item\label{item:TL} (Twist--linearity) Let $c,d \in \mathcal S$ be arbitrary. Then
	\[
	\phi(T_c(d)) = \phi(d) + \pair{d,c} \phi(c) \pmod r,
	\]
	where $\pair{c,d}$ denotes the algebraic intersection pairing and $T_c$ denotes the (left-handed) Dehn twist about $c$.
	\item (Normalization) For $\zeta$ as above, $\phi(\zeta) = 1$.
	\end{enumerate}
	\end{definition}
	
\begin{remark}\label{remark:rdivideschi}
It can be shown that $r$ must divide $2g-2$; this is a consequence, for instance, of the interpretation of $r$-spin structures as roots of the canonical bundle, where it follows from the multiplicativity of degree under tensor product. See also (S3.6) for a more topological perspective. \end{remark}	
	
An essential fact about $r$--spin structures is that they behave predictably on collections of curves bounding an embedded subsurface. This property is called {\em homological coherence}; it is essentially a recasting of the Poincar\'e-Hopf theorem on indices of vector fields. 
	\begin{lemma}[Homological coherence (S3.8)]\label{lemma:homcoh}
	Let $\phi$ be an $r$--spin structure on $\Sigma_g$, and let $S \subset \Sigma_g$ be a subsurface. Suppose $\partial S = c_1 \cup \dots \cup c_k$ and all boundary components $c_i$ are oriented so that $S$ lies to the left. Then
	\[
	\sum_{i = 1}^k \phi(c_i) = \chi(S).
	\]
	\end{lemma}

	It is worth emphasizing that homological coherence implies that for $r > 2$ and an oriented curve $c$, the value $\phi(c)$ is {\em not} determined by the homology class $[c]$ (for instance, if $c$ and $d$ cobound a subsurface of genus $1$, Lemma \ref{lemma:homcoh} shows that $\phi(c) = \phi(d) \pm 2$). Nevertheless, we see in Lemma \ref{lemma:basisvalues} below that an $r$-spin structure does turn out to be {\em determined} by its values on a homological basis. In preparation, we define a {\em geometric homology basis} $\mathcal B = \{x_1, \dots, x_{2g}\}$ to be a collection of oriented simple closed curves whose homology classes are linearly independent and generate $H_1(\Sigma_g;\Z)$.
	
	\begin{lemma}[$r$--spin structures and geometric homology bases (S3.9), (S3.5)]\label{lemma:basisvalues}
	Let
	\[
	\mathcal B = \{x_1, \dots, x_{2g}\}
	\]
	be a geometric homology basis. If $\phi, \psi$ are two $r$--spin structures on $\Sigma_g$ such that $\phi(x_i) = \psi(x_i)$ for $1 \le i \le 2g$, then $\phi = \psi$. 
	
	Conversely, given $\mathcal B$ as above and any vector $v = (v_i) \in (\Z/r\Z)^{2g}$, there exists an $r$--spin structure $\phi$ such that $\phi(x_i) = v_i$ for $1 \le i \le 2g$. 
	\end{lemma}
	
	There is an action of the mapping class group $\Mod(\Sigma_g)$ on the set of $r$--spin structures: for $f \in \Mod(\Sigma_g)$ and $c \in \mathcal S$, define $(f\cdot \phi)(c) = \phi(f^{-1}(c))$. 
 
\begin{definition}[Stabilizer subgroup (S3.14)]\label{definition:stabilizer} Let $\phi$ be a spin structure on a surface $\Sigma_g$. The {\em stabilizer subgroup} of $\phi$, written $\Mod_g[\phi]$, is defined as
\[
\Mod_g[\phi] = \{ f\in \Mod(\Sigma_g) \mid (f \cdot \phi) = \phi\}.
\]
\end{definition}

The simplest class of elements of $\Mod_g[\phi]$ are the Dehn twists that preserve $\phi$. By twist--linearity (Definition \ref{definition:spin}.\ref{item:TL}), if $c$ is a nonseparating curve, $T_c$ preserves $\phi$ if and only if $\phi(c) = 0$. 

\begin{remark}\label{remark:admissorient}
In general, the value $\phi(c)$ depends on the orientation of $c$. However, homological coherence (Lemma \ref{lemma:homcoh}) implies that if $c$ is given the opposite orientation then $\phi(c)$ changes sign. In particular, having $\phi(c)=0$ is a property of {\em unoriented} curves.
\end{remark}

\begin{definition}[Admissible twist, admissible subgroup]\label{definition:admissible}
Let $\phi$ be an $r$--spin structure on $\Sigma_g$. A nonseparating simple closed curve $c$ is said to be {\em $\phi$--admissible} if $\phi(c) = 0$ (if the spin structure $\phi$ is implied, it will be omitted from the notation). The corresponding Dehn twist $T_c \in \Mod_g[\phi]$ is called an {\em admissible twist}. The subgroup
\[
\mathcal T_\phi = \pair{T_c \mid \phi(c) = 0, \,c \text{ nonseparating}} \le \Mod_g[\phi]
\]
is called the {\em admissible subgroup}. 
\end{definition}
	
\subsection{Spin structures from winding number functions}\label{subsection:wnf} The spin structures under study in this paper arise from a construction known as a ``winding number function'' originally due to Chillingworth \cite{Chill_wn1}. We sketch here the basic idea; see \cite{HJ_windingnumber} for details. 

\begin{example}[Winding number function]\label{example:wnf}
Let $\Sigma_g$ be a compact surface endowed with a vector field $V$ with isolated zeroes $p_1, \dots, p_n$ of orders $k_1, \dots, k_n$. Suppose 
\[
\gamma: S^1 \to \Sigma_g \setminus \{p_1, \dots, p_n\}
\]
is a $C^1$--embedded curve on $\Sigma_g \setminus \{p_1, \dots, p_n\}$. Then the winding number of the tangent vector $\gamma'(t)$ relative to $V(\gamma(t))$ determines a $\Z$--valued winding number for $\gamma$. Now if $\gamma'$ is isotopic to $\gamma$ on the surface $\Sigma_g \setminus \{p_1, \dots, \hat{p}_i, \dots, p_n\}$ through $C^1$--embedded curves, then the winding numbers of $\gamma$ and $\gamma'$ differ exactly by $k_i$. Thus, if $r= \gcd(k_1, \dots, k_n)$, the function
\[
wn_{V}: \mathcal S \to \Z/r\Z
\]
is a well--defined map from the set of isotopy classes of oriented curves to $\Z/r\Z$. Both twist--linearity and the fact that $\phi(\zeta)=1$ are easy to check, so in fact $wn_V$ is an $r$--spin structure. 
\end{example}	

Accordingly, we sometimes speak of the value $\phi(c)$ as the ``winding number'' of $c$ even when the $r$-spin structure $\phi$ does not manifestly arise from this construction. (We note in passing that in fact, every $r$-spin structure $\phi$ {\em does} arise from a vector field. This can be seen, e.g. by a direction construction: given the $\phi$-values on curves forming a spine of the surface, it is possible to build a vector field $V$ with the correct winding numbers, and then Lemma \ref{lemma:basisvalues} shows that $wn_V = \phi$.)

\subsection{Classical spin structures and the Arf invariant}\label{subsection:arf} If $r$ is even, then the mod $2$ reduction of $\phi$ determines a ``classical'' spin structure. A basic understanding of the special features present in this case is necessary for a full understanding of $r$--spin structures for $r > 2$ even. In Lemma \ref{lemma:2spinhom} we note the basic fact that bridges our notion of a $2$--spin structure with the classical formulation via quadratic forms. We then proceed to define the ``Arf invariant'' (Definition \ref{definition:arfQF}) and recall some of its basic properties.

\para{From $\mathbf{2}$--spin structures to quadratic forms} For $r >2$, the value of $\phi$ on a simple closed curve $c$ depends on $c$ itself, and not merely the homology class $[c] \in H_1(\Sigma_g;\Z)$. However, the information encoded in a $2$--spin structure is ``purely homological'':
\begin{lemma}[c.f. \cite{Johnson_spin}, Theorem 1A]\label{lemma:2spinhom}
Let $\phi$ be a $2$--spin structure on $\Sigma_g$ and let $c \subset \Sigma_g$ be a simple closed curve. Then $\phi(c) \in \Z/2\Z$ depends only on the homology class $[c] \in H_1(\Sigma_g;\Z/2\Z)$.
\end{lemma}

Following Lemma \ref{lemma:2spinhom}, if $\phi$ is an $r$--spin structure for $r>2$ even, we define the mod $2$ value of $\phi$ on a homology class $z \in H_1(\Sigma_g;\Z/2\Z)$ to be $\phi(c) \pmod 2$ for any simple closed curve $c$ with $[c]=z$. This gives rise to an algebraic structure on $H_1(\Sigma_g;\Z/2\Z)$ known as a {\em quadratic form}. In preparation, recall that if $V$ is a vector space over a field of characteristic $2$, a {\em symplectic form} $\pair{\cdot, \cdot}$ is defined to be a bilinear form satisfying $\pair{v,v} = 0$ for all $v \in V$. 

\begin{definition}\label{definition:QF}
Let $V$ be a vector space over $\Z/2\Z$ equipped with a symplectic form $\pair{\cdot, \cdot}$. A {\em quadratic form} $q$ on $V$ is a function $q:V \to \Z/2\Z$ satisfying
\[
q(x + y) = q(x) + q(y) + \pair{x,y}.
\]
\end{definition}

\begin{remark}\label{remark:SSandQF}
There is a standard correspondence between 2--spin structures and quadratic forms which generalizes for any even $r \ge 2$. If $\phi$ is an $r$--spin structure for $r \ge 2$ even, then the function 
\[
q(x) = \phi(x) + 1 \pmod 2
\]
is a quadratic form on $H_1(\Sigma_g; \Z/2\Z)$; here one evaluates $\phi(x)$ on $x \ne 0$ by choosing a simple closed curve representative for $x$ and applying Lemma \ref{lemma:2spinhom}. 
\end{remark}

\para{Orbits of quadratic forms and the Arf invariant} The symplectic group $\Sp(2g, \Z/2\Z)$ acts on the set of quadratic forms on $H_1(\Sigma_g;\Z/2\Z)$ by pullback. Here we recall the {\em Arf invariant} which describes the orbit structure of this group action. The definition and basic properties presented in this paragraph were developed by Arf in \cite{Arf}.

\begin{definition}[Arf invariant]\label{definition:arfQF}
Let $V$ be a vector space over $\Z/2\Z$ equipped with a symplectic form $\pair{\cdot, \cdot}$, and $q$ be a quadratic form on $V$. The {\em Arf invariant} of $q$, written $\Arf(q)$, is the element of $\Z/2\Z$ defined by
\[
\Arf(q) = \sum_{i = 1}^g q(x_i)q(y_i) \pmod 2,
\]
where $\{x_1, y_1, \dots, x_g, y_g\}$ is {\em any} symplectic basis for $V$. 

For an $r$--spin structure $\phi$ on $\Sigma_g$ with $r$ even, $\Arf(\phi)$ is defined to be the Arf invariant of the quadratic form associated to $\phi$ by Remark \ref{remark:SSandQF}. 

A quadratic form $q$ is said to be {\em even} or {\em odd} according to the parity of $\Arf(q)$. The parity of an $r$--spin structure for $r \ge 2$ even is defined analogously.
\end{definition}

The Arf invariant of a spin structure is easy to compute given any collection of curves which span the homology of the surface. We say that a {\em geometric symplectic basis} for $\Sigma_g$ is a collection
\[
\mathcal B = \{x_1, y_1, \dots, x_g, y_g\}
\]
of $2g$ curves on $S$ such that $i(x_i,y_i) = 1$ for $i = 1, \dots, g$, and such that all other intersections are zero (here $i(c,d)$ denotes the {\em geometric intersection number} of $c,d$). Then $\Arf(\phi)$ may be computed as
\[
\Arf(\phi) = \sum_{i = 1}^g (\phi(x_i)+1)(\phi(y_i)+1) \pmod 2,
\]
where $\mathcal B = \{x_1, y_1, \dots, x_g, y_g\}$ is {\em any} geometric symplectic basis on $\Sigma_g$. 

\begin{remark}\label{remark:arfadd}
The Arf invariant is additive under direct sum; that is, if $V= W_1 \oplus W_2$ where $W_1$ and $W_2$ are symplectically orthogonal and are equipped with nondegenerate quadratic forms $q_1$ and $q_2$, then one has
\[\Arf(q_1 \oplus q_2) = \Arf(q_1) + \Arf(q_2).\]

If $S \subset \Sigma_g$ is a subsurface with one boundary component, then the $r$--spin structure $\phi$ admits an obvious restriction to an $r$--spin structure $\phi\mid_{S}$ on $S$. In this way we speak of the Arf invariant of a subsurface $S$, i.e. $\Arf(\phi\mid_{S})$. If $\Sigma_g = S_1 \cup S_2$ where both subsurfaces have a single boundary component, then the Arf invariant is additive in the obvious sense. The Arf invariant is not defined in any straightforward way on a surface with 2 or more boundary components.
\end{remark}

Since 2--spin structures (or equivalently, quadratic forms on $H_1(\Sigma_g; \Z/2\Z)$) are ``purely homological'' in the sense of Lemma \ref{lemma:2spinhom}, the action of the mapping class group on the set of $2$--spin structures factors through the action of $\Sp(2g,\Z)$ on $H_1(\Sigma_g; \Z)$ and ultimately through $\Sp(2g,\Z/2\Z)$ acting on $H_1(\Sigma_g;\Z/2\Z)$. Thus there is an algebraic counterpart to the notion of spin structure stabilizer defined in Definition \ref{definition:stabilizer}.

\begin{definition}[Algebraic stabilizer subgroup]\label{definition:algstab}
Let $q$ be a quadratic form on $H_1(\Sigma_g; \Z/2\Z)$. The {\em algebraic stabilizer subgroup} is the subgroup
\[
\Sp(2g,\Z/2\Z)[q] = \{A \in \Sp(2g, \Z/2\Z) \mid A \cdot q = q\}.
\]
We define the algebraic stabilizer subgroup $\Sp(2g, \Z)[q]$ as the preimage of $\Sp(2g,\Z/2\Z)[q]$ in $\Sp(2g,\Z)$ under the usual quotient map.
\end{definition}

The Arf invariant of a quadratic form is invariant under the action of $\Sp(2g,\Z/2\Z)$ (and hence under the action of $\Sp(2g,\Z)$ and $\Mod(\Sigma_g)$), and in fact this is the only invariant of the $\Mod(\Sigma_g)$ action.

More generally, for any even $r$ the $\Mod(\Sigma_g)$ action on the set of $r$--spin structures must always preserve the induced Arf invariant, and as above, this is the only invariant of the $\Mod(\Sigma_g)$ action. This fact was originally proven by Randal--Williams \cite[Theorem 2.9]{RW}.

\begin{lemma}[c.f. (S4.2), (S4.9)]\label{lemma:stabindex}
Let $\Sigma_g$ be a closed surface of genus $g \ge 2$ and let $r$ divide $2g-2$. If $r$ is odd, then the mapping class group acts transitively on the set of $r$--spin structures. If $r$ is even, then there are two orbits of the $\Mod(\Sigma_g)$ action, distinguished by their Arf invariant.

Consequently, if $\phi$ is an $r$--spin structure, then the index $[\Mod(\Sigma_g) : \Mod_g[\phi]]$ is
\begin{itemize}
\item $r^{2g}$ if $r$ is odd, 
\item $(r/2)^{2g} \left( 2^{g-1} (2^g + 1) \right)$ if $r$ is even and $\phi$ has even Arf invariant, and
\item $(r/2)^{2g} \left( 2^{g-1} (2^g - 1) \right)$ if $r$ is even and $\phi$ has odd Arf invariant.
\end{itemize}
\end{lemma}

\section{The sliding principle}\label{section:sliding}

This section is devoted to establishing a versatile lemma for computations in the mapping class group which we call the {\em sliding principle}. In the course of our later work in Section \ref{section:finitegen}, we will often need to demonstrate that given a subgroup $\Gamma \le \Mod(\Sigma_g)$ and two simple closed curves $a$ and $b$, there is some $\gamma \in \Gamma$ such that $\gamma(a) = b$. The statements of the relevant lemmas (\ref{lemma:hardeven}, \ref{lemma:acurve}, and \ref{lemma:hardodd}) are technical, and their proofs are necessarily computational. However, they are all manifestations of the sliding principle, which appears as Lemma \ref{lemma:sliding} below as the culmination of a sequence of examples. 

Ultimately the sliding principle is merely a labor-saving device which circumvents the need for lengthy explicit Dehn twist computations, and this section can safely be taken as a black box by those more interested in the global structure of the argument and less so the minutiae of mapping class group computations.

\subsection{Sliding along chains and Birman--Hilden theory}\label{subsection:chainslide}

The simplest example of the sliding principle is the braid relation: recall that if $a$ and $b$ are simple closed curves on a surface which intersect exactly once, then 
\[T_aT_bT_a=T_bT_aT_b\]
and this element interchanges the curves $a$ and $b$. More generally, if $(a_1, \ldots, a_n)$ is an {\em $n$--chain} of simple closed curves (i.e. curves $a_i$ and $a_{i+1}$ intersect transversely once, and $a_i$ and $a_j$ are disjoint for $\abs{i-j} >1$), then there is an element of $\Gamma := \langle T_{a_1}, \ldots, T_{a_n} \rangle$ which takes $a_1$ to $a_n$. We think of the curve $a_1$ as ``sliding'' along the chain $(a_1, \dots, a_n)$ to $a_n$.

The theory of Birman and Hilden (see, e.g., \cite{MargWin_BHSurvey}) clarifies this phenomenon by identifying the group $\Gamma$ as a braid group. This identification provides an explicit model for the action of $\Gamma$ on simple closed curves, making the above statement apparent.

Namely, let $W$ be a regular neighborhood of the $n$-chain $\{a_1, \ldots, a_n\}$; then $W$ has a unique hyperelliptic involution $\iota$ which (setwise) fixes each $a_i$ curve.
The quotient $W/ \langle \iota \rangle$ is a disk with $n+1$ marked points, and the Birman--Hilden theorem implies that 
\begin{equation}\label{eqn:BH}
C_W(\iota) \cong B_{n+1}
\end{equation}
where $C_W(\iota)$ is the centralizer of $\iota$ inside of $\Mod(W)$. The Dehn twist about $a_i$ descends to the half--twist $h_i$ interchanging the $i^\text{th}$ and $(i+1)^\text{st}$ curves, and so we see that under the isomorphism \eqref{eqn:BH}, $\{ T_{a_i}, \ldots, T_{a_n}\}$ corresponds to the standard Artin generators for $B_{n+1}$.

Now in $B_{n+1}$ it is evident that any two half--twists $h_i$ and $h_j$ are conjugate, for example, by a braid which interchanges the $i^\text{th}$ and $(i+1)^\text{st}$ strands with the $j^\text{th}$ and $(j+1)^\text{st}$ strands. By the Birman--Hilden correspondence, $T_{a_i}$ and $T_{a_j}$ are conjugate in $C_W(\iota)$, and hence there is some element of $\Gamma$ taking $a_i$ to $a_j$ (and vice--versa).

Similarly, any two sub--braid groups $B_{i,j}:= \langle h_i, h_{i+1} \ldots, h_j \rangle$ and $B_{k, \ell}:= \langle h_k, h_{k+1} \ldots, h_\ell \rangle$
generated by consecutive half--twists are conjugate in $B_{n+1}$ if and only if $j-i = \ell - k$, that is, if they act on the same number of strands. 
In terms of subsurfaces, this means that if $Y_{i,j}$ and $Y_{k, \ell}$ denote the subsurfaces given as regular neighborhoods of $(a_i, \ldots, a_j)$ and $(a_k, \ldots, a_\ell)$, respectively, then there is some element $\gamma \in \Gamma$ which identifies the chains in an order--preserving way and hence takes $Y_{i,j}$ to $Y_{k, \ell}$.

The sliding principle for chains then boils down to using this action to transport curves living on $Y_{i,j}$ to curves on $Y_{k, \ell}$. In order to make this work, we need a coherent way of marking each subsurface.

By construction, $ Y_{i,j} \setminus (a_i \cup \ldots \cup a_j)$ is a union of either one or two annuli, one for each component of $\partial  Y_{i,j}$. In particular, the chain $(a_i, \ldots, a_j)$ determines a marking of $Y_{i,j}$ up to mapping classes of $Y_{i,j}$ preserving each curve of the chain. In the case at hand, the only such elements are Dehn twists about $\partial Y_{i,j}$ and the hyperelliptic involution $\iota$ (since $\iota$ fixes each curve of the subchain $(a_i, \ldots, a_j)$, it restricts to an involution of the regular neighborhood $Y_{i,j}$).

Choose an orientation on $a_1$; this specifies an orientation on each subsequent $a_i$ by the convention that $\pair{a_i, a_{i+1}} = 1$. Now the hyperelliptic involution reverses the orientation of each $a_i$, and hence the data of $(a_i, \ldots, a_j)$ together with their orientations is enough to determine a marking up to twists about $\partial Y_{i,j}$. Of course, the same procedure may be repeated for $Y_{k, \ell}$.

The identification $\gamma(Y_{i,j}) = Y_{k, \ell}$ should therefore be thought of as an identification of {\em marked} subsurfaces (up to twisting about $\partial Y_{i,j}$), and so can be used to transport any simple closed curve $c$ supported on $Y_{i,j}$ to a curve $\gamma(c)$ supported on $Y_{k, \ell}$. Moreover, one can use the (signed) intersection pattern of $c$ with the $a_i$ to explicitly identify $\gamma(c)$ as a curve on $Y_{k, \ell}$.

\begin{example}\label{example:chainslide}
As a simple example of the sliding principle, consider the curves $x$ and $y$ shown in Figure \ref{fig:sliding2} below. The curve $x$ is supported on the (subsurface determined by the) $5$--chain $(5,4,3,6,7)$, and $y$ is supported on $(3,6,7,8,9)$. When $(5,4,3,6,7)$ is slid to $(3,6,7,8,9)$, this identification takes $x$ to $y$. 
\end{example}

\begin{remark}
A similar philosophy can be used to investigate the $\Gamma$ action on curves which merely intersect $W$, but then one must be careful to take into account the incidence of the curve with $\partial W$ and ensure that there is no twisting about $\partial Y_{i,j}$ (c.f. \cite[Lemmas A.4--7]{Calderon_strata}).
\end{remark}

\subsection{General sliding}\label{subsection:generalslide}
So far, what we have discussed is just an extended consequence of the Birman--Hilden correspondence for a hyperelliptic subsurface. The general sliding principle is a method for investigating the action on a union of such subsurfaces.

Let $\mathcal{C}$ be a set of simple closed curves on the surface $\Sigma_g$ and set
\[\Gamma := \langle T_a : a \in \mathcal{C} \rangle.\]
Define the {\em intersection graph} $\Lambda_\mathcal{C}$ of $\mathcal{C}$ to have a vertex for each curve of $\mathcal{C}$, and two vertices to be connected by an edge if and only if the curves they represent intersect exactly once. Without loss of generality, we will assume that $\Lambda_\mathcal{C}$ is connected (otherwise each component can be dealt with separately).

Paths in the intersection graph $\Lambda_{\mathcal C}$ correspond to chains on the surface, which in turn fill hyperelliptic subsurfaces. By the discussion above, the $\Gamma$ action can be used to slide curves supported in a neighborhood of $\mathcal{C}$ along paths in the intersection graph.

Generally, however, a curve cannot traverse all of $\Lambda_\mathcal{C}$ just by sliding. In particular, the subsurface carrying the curve can only transfer between chains or reverse the order of its associated chain when there is enough space for it to ``turn around.'' For example, consider the set of curves $\mathcal{C} \subset \Sigma$ shown in Figure \ref{fig:sliding2}, whose intersection graph $\Lambda_\mathcal{C}$ is a tripod with legs of length 2, 2, and 6. We claim that $\Gamma$ acts transitively on the set of (ordered) 3--chains in $\mathcal{C}$.

\begin{figure}[ht]
\centering
\labellist
\small
\pinlabel $\mathcal{C}$ [tl] at 230.2 15.2
\pinlabel $\Lambda_{\mathcal{C}}$ [tl] at 375 15.2
\pinlabel 1 [br] at 43.2 57
\pinlabel 2 [bl] at 76.6 46.8
\pinlabel 3 [bl] at 94.4 40.8
\pinlabel 4 [tl] at 69.4 28
\pinlabel 5 [tl] at 31.2 12.8
\pinlabel 6 [b] at 111.8 38
\pinlabel 7 [b] at 131.8 41
\pinlabel 8 [b] at 151.2 38
\pinlabel 9 [b] at 174.2 40
\pinlabel 1 [tr] at 302 69.4
\pinlabel 2 [tr] at 314.8 57.2
\pinlabel 3 [b] at 320.8 35.8
\pinlabel 4 [tl] at 295.8 27
\pinlabel 5 [tr] at 292.4 13
\pinlabel 6 [b] at 340.8 35.8
\pinlabel 7 [b] at 360.8 35.8
\pinlabel 8 [b] at 380.8 35.8
\pinlabel 9 [b] at 400.8 35.8
\pinlabel $x$ at 90 7
\pinlabel $y$ at 141 15
\endlabellist
\includegraphics{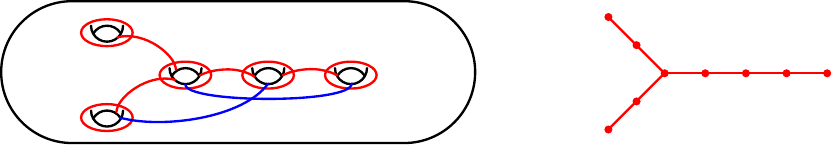}
\caption{A set of simple closed curves $\mathcal{C}$ and its intersection graph $\Lambda_\mathcal{C}$. The $\Gamma$-action is transitive on $3$--chains but not on $5$--chains (at least not obviously so).}
\label{fig:sliding2}
\end{figure}

Indeed, given any 3--chain in $\mathcal{C}$, the sliding principle for chains implies that it can be taken to either
$(a_1, a_2, a_3)$ or $(a_7, a_8, a_9)$,
possibly with orientation reversed. The chains $(a_1, a_2, a_3)$ and $(a_7, a_8, a_9)$ are in turn related by sliding, so $\Gamma$ acts transitively on the set of unordered 3--chains. Therefore, to see that $\Gamma$ acts transitively on ordered 3--chains, it suffices to show that $(a_1, a_2, a_3)$ is in the $\Gamma$ orbit of $(a_3, a_2, a_1)$. This follows by repeated sliding:
\[(a_1, a_2, a_3) \sim (a_3, a_4, a_5) \sim (a_9, a_8, a_7) \sim (a_3, a_2, a_1)\]
where we have written $c \sim c'$ to indicate that the chain $c$ can be slid to the chain $c'$ along a chain in $\mathcal{C}$.

However, the sliding principle does not imply that $\Gamma$ acts transitively on the set of 5--chains in $\mathcal{C}$. This can be explained by a lack of space in $\Lambda_\mathcal{C}$: the 5--chain $(a_1, \ldots, a_5)$ cannot be slid to lie entirely on one branch of $\Lambda_{\mathcal{C}}$, and so we cannot perform the same turning maneuvers as in the case of 3--chains.

We record this intuition in the following statement, the proof of which is just a repeated application of the sliding principle for chains.

\begin{lemma}[The sliding principle]\label{lemma:sliding}
Suppose that $\mathcal{C}$ is a set of simple closed curves on a surface $\Sigma_g$ and set
\[\Gamma = \langle T_a : a \in \mathcal{C} \rangle.\]
Let $Y, W \subset \Sigma_g$ be subsurfaces filled by $k$-chains $(y_1, \ldots, y_k) \subset \mathcal{C} $ and $(w_1, \ldots, w_k) \subset \mathcal{C} $, respectively.
If there exists a sequence $c_1, \ldots, c_m$ of chains in $\mathcal{C}$ (not necessarily of the same length) such that
\begin{itemize}
\item $(y_1, \ldots, y_k)$ is a subchain of $c_1$
\item $(w_1, \ldots, w_k)$  is a subchain of $c_m$
\item $c_i$ and $c_{i+1}$ overlap in a subchain of length at least $k$
\end{itemize}
then there exists $\gamma \in \Gamma$ taking $(y_1, \ldots, y_k)$ to $(w_1, \ldots, w_k)$. Moreover, $\gamma$ induces a natural identification of the simple closed curves supported entirely on $Y$ with those supported entirely on $W$.
\end{lemma}

\section{Finite generation of the admissible subgroup}\label{section:finitegen}

\subsection{Outline of Theorem \ref{theorem:genset}}

We now turn to the proof of Theorem \ref{theorem:genset}. As the proof is spread out over the next two sections, both of which contain rather technical lemmas, we pause here to remind the reader of the work remaining to be done (see also the outline given in Section \ref{section:introduction}). 

At the highest level, the proof divides into two pieces: we first establish the ``maximal'' case $r  = 2g-2$ formulated in Theorem \ref{theorem:genset}.\ref{case:hardeven} and \ref{theorem:genset}.\ref{case:hardodd}, and then we will use this to establish the case of general $r$ as formulated in Theorem \ref{theorem:genset}.\ref{case:r}.

The proof of the maximal case $r = 2g-2$ divides further into two steps. The first step, carried out in Section \ref{section:finitegen}, shows that the finite collection of twists described in Theorem \ref{theorem:genset} generates the full {\em admissible subgroup} $\mathcal T_\phi$ (c.f. Definition \ref{definition:admissible}). The second step (Proposition \ref{prop:admissfull}) is to show that the admissible subgroup coincides with the spin structure stabilizer: $\mathcal T_\phi = \Mod_g[\phi]$. This is accomplished in Section \ref{section:admiss=stab} (more precisely, Sections \ref{subsection:firststep}--\ref{subsection:johnson}). The work here applies to general $r$ with essentially no modification, and in anticipation of the general case, we formulate and prove Proposition \ref{prop:admissfull} for arbitrary $r$.

Given the maximal case, the proof in the general case is actually quite easy, and is handled in Section \ref{section:generalr}. In light of Proposition \ref{prop:admissfull}, it suffices to show that a finite collection of twists as given in Theorem \ref{theorem:genset}.\ref{case:r}, together with the stabilizer of a lift of $\phi$, generates the admissible subgroup $\T_\phi$.

\subsection{Statement of the main result of Section \ref{section:finitegen}} The first step is to show that each of the finite collections of Dehn twists presented in Figures \ref{figure:hardcaseeven} and \ref{figure:hardcaseodd} generate their respective admissible subgroups $\mathcal T_\phi$. This is the main result of the next  section.

\begin{remark}
In the statement of Theorem \ref{theorem:genset}, we consider both an $r$-spin structure $\phi$ for arbitrary $r$ dividing $2g-2$, as well as a lift of $\phi$ to a $(2g-2)$-spin structure, which we notate $\tilde \phi$. In this section, all spin structures under consideration have $r = 2g-2$, and will be denoted $\phi$ for ease of notation.
\end{remark}
 
	\begin{proposition}\label{lemma:gammaadmiss}
	In case \ref{case:hardeven} (respectively, case \ref{case:hardodd}) of Theorem \ref{theorem:genset}, let $\Gamma$ denote the group generated by the indicated collections of Dehn twists. Then $\Gamma = \mathcal T_\phi$, where $\phi$ is the $(2g-2)$--spin structure specified by assigning $\phi(c) = 0$ for every curve $c$ appearing in Figure \ref{figure:hardcaseeven} (respectively, Figure \ref{figure:hardcaseodd}).
	\end{proposition}

The proof of Proposition \ref{lemma:gammaadmiss} is completed in Section \ref{subsection:containspin} as the synthesis of a series of technical lemmas. In Section \ref{subsection:subspin} we recall the notion of a ``spin subsurface push subgroup'' $\tilde \Pi(b)$ from \cite{Salter_monodromy} and establish a criterion (Lemma \ref{lemma:pushmakesT}) for $\Gamma$ to contain $\mathcal T_\phi$ in terms of $\tilde \Pi(b)$. In Section \ref{subsection:networks}, we review the theory of ``networks'' from \cite{Salter_monodromy}, and use this to formulate an explicit generating set for $\tilde \Pi(b)$ (Lemma \ref{lemma:networkgens}). In Section \ref{section:relations}, we briefly recall some relations in the mapping class group. Finally in Section \ref{subsection:containspin} we use the results of the preceding sections to show the containment $\tilde \Pi(b) \le \Gamma$, and so conclude the proof of Proposition \ref{lemma:gammaadmiss}.

As our ultimate goal is the proof of Proposition \ref{lemma:gammaadmiss}, throughout this section we consider only $(2g-2)$--spin structures. Refer to \cite{Salter_monodromy} for the corresponding statements for general $r$.

\subsection{Spin subsurface push subgroups}\label{subsection:subspin}
Here we recall the notion of a ``spin subsurface push subgroup'' from \cite[Section 8]{Salter_monodromy}. The main objective of this subsection is Lemma \ref{lemma:pushmakesT} below, which provides a criterion for a subgroup $H \le \Mod(\Sigma_g)$ to contain the admissible subgroup $\mathcal T_\phi$ in terms of a spin subsurface push subgroup. Let $\Sigma_g$ be a closed surface equipped with a $(2g-2)$--spin structure $\phi$, and let $b \subset \Sigma_g$ be an essential, oriented, nonseparating curve satisfying $\phi(b) = -1$. Define $S'$ to be the {\em closed} subsurface of $\Sigma_g$ obtained by removing an open annular neighborhood of $b$; let $\Delta$ denote the boundary component of $S'$ corresponding to the left side of $b$. Let $\overline{S'}$ denote the surface obtained from $S'$ by capping off $\Delta$ by a disk. 

Combining a suitable form of the Birman exact sequence (c.f. \cite[Section 4.2.5]{FarbMarg}) with the inclusion homomorphism $i_*: \Mod(S') \to \Mod(\Sigma_g)$, the capping operation induces a homomorphism 
\[
\mathcal P: \pi_1(UT\overline{S'}) \to \Mod(\Sigma_g);
\]
here $UT \overline{S'}$ denotes the unit tangent bundle to $\overline S'$. Intuitively, $\mathcal P$ acts as follows: suppose that $c \in UT \overline{S'}$ can be represented in $\Mod(S')$ by pushing $\Delta$ along some simple path $\gamma$. Then $c$ has an expression in terms of Dehn twists: 
\[
c = T_{\gamma_L} T_{\gamma_R}^{-1} T_\Delta^k
\]
for some $k \in \Z$, where $\gamma_L, \gamma_R$ are the simple closed curves on $S'$ lying to the left (resp. right) of the path $\gamma$. Each of the curves $\gamma_L, \gamma_R, \Delta$ are contained in $\Sigma_g$ under the inclusion $S' \into \Sigma_g$, and so the above expression determines the mapping class $\mathcal P(c) \in \Mod(\Sigma_g)$.

We call the image 
\[
 \Pi(b) := \mathcal P(\pi_1(UT\overline{S'}))
\]
a {\em subsurface push subgroup}\footnote{We have attempted to improve the notation introduced in \cite[Section 8]{Salter_monodromy} where the corresponding subsurface push subgroup was denoted $\Pi(S', \Delta)$.} and remark that $\mathcal P$ can be shown to be an injection.

\begin{definition}[Spin subsurface push subgroup]
Let $\Sigma_g$ be a closed surface equipped with a $(2g-2)$--spin structure $\phi$, and let $b \subset \Sigma_g$ be an essential, oriented, nonseparating curve satisfying $\phi(b) = -1$. The {\em spin subsurface push subgroup} $\tilde \Pi(b)$ is\footnote{Again, the notation here differs slightly with \cite{Salter_monodromy}, where the spin subsurface push subgroup is denoted $\tilde \Pi(\Sigma_g \setminus \{b\})$.} the intersection
\[
\tilde \Pi(b) : = \Pi(b) \cap \Mod_g[\phi].
\]
\end{definition}

\begin{lemma}[\cite{Salter_monodromy}, Lemma 8.1]\label{lemma:sspschar}
The spin subsurface push subgroup $\tilde \Pi( b)$ is a finite--index subgroup of $\Pi(b)$. It is characterized by the group extension 
\begin{equation}\label{equation:pi}
1 \to \pair{T_b^{2g-2}} \to \tilde \Pi(b) \to \pi_1(\overline{S'}) \to 1;
\end{equation}
the map $\tilde \Pi(b) \to \pi_1(\overline{S'})$ is induced by the capping map $S' \to \overline{S'}$ where the boundary component corresponding to the left side of $b$ is capped off with a punctured disk. 
\end{lemma}

The following Lemma \ref{lemma:pushmakesT} was established in \cite{Salter_monodromy}. It shows that a spin subsurface push subgroup $\tilde \Pi(b)$ is ``not far'' from containing the entire admissible subgroup $\mathcal T_\phi$. In the next subsection, we will make this more concrete by finding an explicit finite set of generators for $\tilde \Pi(b)$, and in Section \ref{subsection:containspin} we will do the work necessary to show that $\Gamma$ contains this generating set, and consequently to show the equality $\Gamma = \mathcal T_\phi$.

	\begin{lemma}[C.f. \cite{Salter_monodromy}, Proposition 8.2]\label{lemma:pushmakesT}
	Let $\phi$ be a $({2g-2})$--spin structure on a closed surface $\Sigma_g$ for $g \ge 5$. Let $(a,a',b)$ be an ordered $3$-chain of curves with $\phi(a) = \phi(a') = 0$ and $\phi(b) = -1$. Let $H \leqslant \Mod(\Sigma_g)$ be a subgroup containing $T_{a}, T_{a'}$ and the spin subsurface push group $\widetilde\Pi(b)$. Then $H$ contains $\mathcal T_\phi$.
	\end{lemma}

\subsection{Networks} \label{subsection:networks}
In this subsection we describe an explicit finite generating set for $\widetilde \Pi(b)$, stated as Lemma \ref{lemma:networkgens}. This is formulated in the language of ``networks'' from \cite[Section 9]{Salter_monodromy}. 

\begin{definition}[Networks]\label{definition:network}
Let $S = \Sigma_{g,b}^n$ be a surface, viewed as a compact surface with marked points. A {\em network} on $S$ is any collection $\mathcal N = \{a_1,\dots, a_n\}$ of simple closed curves on $S$, disjoint from any marked points, such that $\#(a_i \cap a_j) \le 1$ for all pairs of curves $a_i, a_j \in \mathcal N$, and such that there are no triple intersections. A network $\mathcal N$ has an associated {\em intersection graph} $\Lambda_{\mathcal N}$, whose vertices correspond to curves $x \in \mathcal N$, with vertices $x,y$ adjacent if and only if $\#(x\cap y) = 1$.  A network is said to be {\em connected} if $\Lambda_{\mathcal N}$ is connected, and {\em arboreal} if $\Lambda_{\mathcal N}$ is a tree. A network is {\em filling} if 
\[
S \setminus \bigcup_{a \in \mathcal N} a
\]
is a disjoint union of disks and boundary-parallel annuli; each disk component is allowed to contain at most one marked point of $S$ and each annulus component may not contain any.
\end{definition}

The following lemma provides the promised explicit finite generating set for (a supergroup of) $\widetilde \Pi(b)$. As always, we assume that $\phi$ is a $({2g-2})$--spin structure on $\Sigma_g$ with $g \ge 5$. Let $b \subset \Sigma_g$ be an essential, oriented, nonseparating curve satisfying $\phi(b) = -1$, and consider the surface $\overline{S'}$ of Section \ref{subsection:subspin} as well as the spin subsurface push subgroup $\widetilde \Pi(b)$.
\begin{lemma}\label{lemma:networkgens}
 Suppose $\mathcal N$ is an arboreal filling network on $\overline{S'}$, and suppose that there exist $a, a' \in \mathcal N$ such that $a \cup a' \cup b$ forms a pair of pants on $\Sigma_g$.  Let $H \leqslant \Mod(\Sigma_g)$ be a subgroup containing $T_a$ for each $a \in \mathcal N$ and $T_b^{2g-2}$. Then $\widetilde \Pi(b) \leqslant H$.
\end{lemma}
\begin{proof}
This is an amalgamation of the results of \cite[Section 9]{Salter_monodromy}. See especially the proof of \cite[Theorem 9.5]{Salter_monodromy} as well as \cite[Lemma 9.4]{Salter_monodromy}. In the latter, we have replaced the hypothesis $P(a_1) \in H$ by the requirement that $a \cup a' \cup b$ form a pair of pants; in this case, the corresponding push map is given simply by $T_a T_{a'}^{-1} \in H$. 
\end{proof}

\subsection{Relations in the mapping class group}\label{section:relations} In preparation for the explicit computations to be carried out in Section \ref{subsection:containspin}, we collect here some relations within the mapping class group. The chain and lantern relations are classical; a discussion of the $D$ relation can be found in, e.g., \cite[Section 2.3]{Salter_monodromy}. Throughout, all Dehn twists are taken to be left-handed.

\begin{lemma}[The chain relation]\label{lemma:chain}
Let $(a_1, \dots, a_k)$ be a chain of simple closed curves. If $k$ is even, let $d$ denote the single boundary component of the subsurface determined by $(a_1, \dots, a_k)$, and if $k$ is odd, let $d_1, d_2$ denote the two boundary components. 
\begin{itemize}
\item If $k$ is even, then $T_d = (T_{a_1} \dots T_{a_k})^{2k+2}$.
\item If $k$ is odd, then $T_{d_1} T_{d_2} = (T_{a_1} \dots T_{a_k})^{k+1}$.
\end{itemize}
\end{lemma}

\begin{lemma}[The lantern relation]\label{lemma:lantern}
Let $a,b,c,d,x,y,z$ be the simple closed curves shown in Figure \ref{figure:lantern}. Then
\[
T_a T_b T_c T_d = T_x T_y T_z.
\]
\end{lemma}

\begin{figure}[ht]
\labellist
\small
\pinlabel $a$ at 45 57
\pinlabel $b$ at 78 113 
\pinlabel $c$ at 111 57
\pinlabel $d$ at 140 20
\pinlabel $x$ at 35 100
\pinlabel $y$ at 120 100
\pinlabel $z$ at 78 25
\endlabellist
\includegraphics{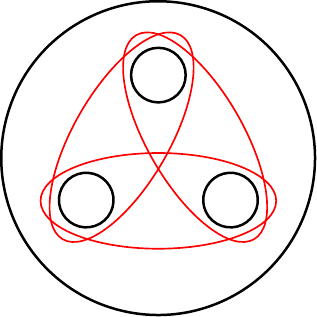}
\caption{The lantern relation.}
\label{figure:lantern}
\end{figure}

\begin{figure} 
\labellist
\small
\pinlabel $\Delta_0$ [tl] at 17 13.6
\pinlabel $\Delta_1$ [l] at 312 96
\pinlabel $\Delta_1'$ [l] at 312 24
\pinlabel $\Delta_2$ [bl] at 275 100
\pinlabel $a$ [l] at 60 96
\pinlabel $a'$ [l] at 60 24
\pinlabel $c_1$ [tr] at 40.8 48
\pinlabel $c_2$ [bl] at 85 70
\pinlabel $c_3$ [tr] at 140 42
\pinlabel $c_4$ [bl] at 160 70
\pinlabel $c_{2g-1}$ [tr] at 245 42
\pinlabel $c_{2g}$ [br] at 270 70
\endlabellist
\includegraphics{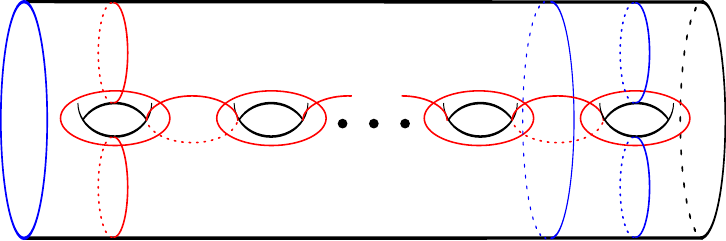}
\caption{The configuration of curves used in the $D$ relation.}
\label{figure:dnrel}
\end{figure}

\begin{lemma}[The $D$ relation]\label{lemma:dn} 
Let $n \ge 3$ be given, and express $n = 2g+1$ or $n = 2g+2$ according to whether $n$ is odd or even. With reference to Figure \ref{figure:dnrel}, let $H_n$ be the group generated by elements of the form $T_x$, with $x \in \mathscr D_n$ one of the curves below:
\begin{align*}
\mathscr D_{n} &= \{a, a', c_1, \dots, c_{n-2}\}.
\end{align*}
Then for $n = 2g+1$ odd,
\[
T_{\Delta_0}^{2g-1} T_{\Delta_2} \in H_n,
\]
and for $n = 2g+2$ even,
\[
T_{\Delta_0}^{g} T_{\Delta_1} T_{\Delta_1'}\in H_n.
\]
\end{lemma}
	
\subsection{Generating the spin subsurface push subgroup}\label{subsection:containspin} 
In this section we complete the proof of Proposition \ref{lemma:gammaadmiss}. We begin by establishing the containment of certain Dehn twists in the group $\Gamma$ generated by the collections of Dehn twists indicated in cases \ref{case:hardeven} and \ref{case:hardodd} of Theorem \ref{theorem:genset}, which will in turn allow us to apply the machinery developed in the previous sections.

\para{Case \ref{case:hardeven}}
With reference to Figure \ref{figure:hardeven2}, we observe that the $3$--chain $(a_4,a_3,b)$ satisfies the hypotheses of the $3$--chain $(a,a',b)$ of Lemma \ref{lemma:pushmakesT} and that $\{T_{a_i} \mid 0 \le i \le 2g-1,\ i \ne 3\}$ forms a suitable arboreal filling network on the capped surface $\overline{S'}$. Therefore, in order to apply Lemmas \ref{lemma:pushmakesT} and \ref{lemma:networkgens}, we need only to show that $T_b^{2g-2} \in \Gamma$.

\begin{figure}[ht]
\labellist
\small
\pinlabel $b$ at 32 16
\pinlabel $a_0$ [r] at 113 16
\pinlabel $a_1$ [b] at 103 78
\pinlabel $a_2$ [bl] at 100 50
\pinlabel $a_4$ at 70 48
\pinlabel $a_3$ [br] at 25.6 46.4
\pinlabel $a_5$ [bl] at 125 18
\pinlabel $a_6$ at 154 48
\pinlabel $a_{2g-1}$ [b] at 315.2 48.8
\pinlabel $c$ at 75 16
\pinlabel $c''$ at 135 70
\pinlabel $c'$ at 78 95
\endlabellist
\includegraphics{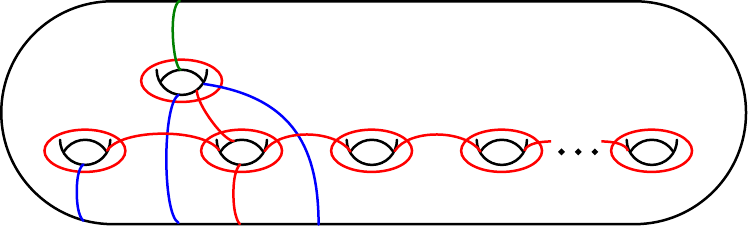}
\caption{The configuration of curves defining $\Gamma$ in case \ref{case:hardeven}, along with the auxiliary curves $b, c, c', c''$. For clarity, the portions of the curves on the back side have been omitted; in all cases, they continue on the back as the mirror image (relative to the plane of the page).}
\label{figure:hardeven2}
\end{figure}

\begin{lemma} \label{lemma:hardeven}
Let $b$ be the curve indicated in Figure \ref{figure:hardeven2}. Then $T_b^{2g-2} \in \Gamma$ in case \ref{case:hardeven} of Theorem \ref{theorem:genset}.
\end{lemma}
\begin{proof}
We will first show that $T_c^{2g-2} \in \Gamma$ for the curve $c$ shown in Figure \ref{figure:hardeven2}, and then we will conclude the argument by showing that $b$ and $c$ are in the same orbit of the $\Gamma$--action on simple closed curves. Consider the $\mathscr D_{2g-3}$--configuration determined by the curves $a_0, a_2, a_5, a_6, \dots, a_{2g-1}$ with boundary components $c, c'$. Applying the $D$ relation (Lemma \ref{lemma:dn}),
\[
T_c^{2g-5} T_{c'} \in \Gamma.
\]
Next, consider the $\mathscr D_{5}$--configuration determined by $a_0,a_2, a_3, a_4, a_5$ with boundary components $c', c''$. Applying the $D$ relation to this, we find
\[
T_{c'} T_{c''}^3 \in \Gamma.
\]
Finally, the chain relation as applied to $(a_0,a_5, a_2)$ shows that 
\[
T_c T_{c''} \in \Gamma;
\]
combining these three results shows $T_c^{2g-2} \in \Gamma$. 

The curves $b,c$ are boundary components of the $3$--chains $(a_4,a_5,a_0)$ and $(a_2, a_5,a_0)$, respectively. Since we can slide the $3$--chains to each other via
\[(a_4,a_5,a_0) \sim (a_3, a_4, a_5) \sim (a_5, a_6, a_7) \sim (a_1, a_2,a_5) \sim (a_2, a_5,a_0)\]
the sliding principle (Lemma \ref{lemma:sliding}) shows that $b$ can be taken to $c$ by an element of $\Gamma$. 
\end{proof}

\para{Case \ref{case:hardodd}}
In Case \ref{case:hardodd}, the network we use does not consist entirely of the curves $a_0, \dots, a_{2g-1}$, and so our first item of business is to see that $\Gamma$ contains the admissible twists $T_{a_{2g}}, T_{a_{2g+1}}$ shown in Figure \ref{figure:hardodd3}. In Lemma \ref{lemma:hardodd}, we will also need to use the twist $T_{a_{2g+2}}$ for the curve $a_{2g+2}$ shown in Figure \ref{figure:a}, and in fact we will obtain $T_{a_{2g}}, T_{a_{2g+1}} \in \Gamma$ from the containment $T_{a_{2g+2}} \in \Gamma$ established in Lemma \ref{lemma:acurve}.

\begin{figure}[ht]
\labellist
\small
\pinlabel $a_{2g+1}$ [r] at 167.2 83.2
\pinlabel $a_1$ [b] at 52 66.4
\pinlabel $a_2$ [b] at 74 60.8
\pinlabel $a_{2g-1}$ [b] at 321 65.6
\pinlabel $b$ [r] at 210 83.2
\pinlabel $a_{2g}$ at 210 13
\pinlabel $a_8$ [b] at 220 60.8
\pinlabel $a_9$ [b] at 240 65.6
\endlabellist
\includegraphics{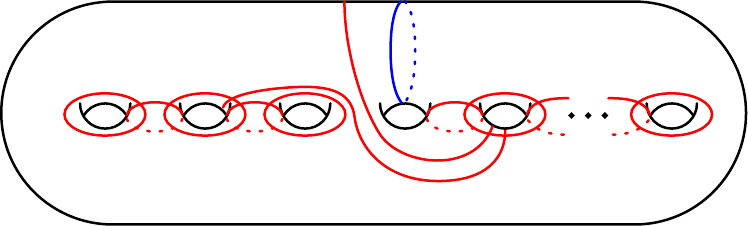}
\caption{The filling arboreal network $\mathcal N$ used in Case \ref{case:hardodd}. Curves $a_{2g}, a_{2g+1}$ continue on the back of the surface as the mirror image.}
\label{figure:hardodd3}
\end{figure}

\begin{figure}
\includegraphics{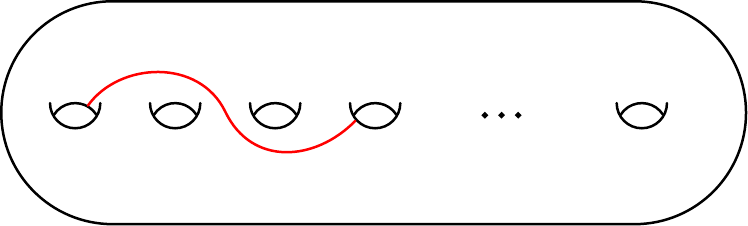}
\caption{The curve $a_{2g+2}$ used in the proof of Lemma \ref{lemma:hardodd}. It continues on the back side as the mirror image.}
\label{figure:a}
\end{figure}

\begin{lemma} \label{lemma:acurve}
In case \ref{case:hardodd} of Theorem \ref{theorem:genset}, we have $T_{a_{2g+2}} \in \Gamma$ for the curve $a_{2g+2}$ shown in Figure \ref{figure:a}.
\end{lemma}
\begin{proof}
This is closely related to the sliding principle. One verifies (see Figure \ref{figure:acomp}) that
\[
(T_{a_5}T_{a_4}T_{a_3}T_{a_2})(T_{a_6}T_{a_5}T_{a_4}T_{a_3})(T_{a_7}T_{a_6}T_{a_5}T_{a_4})(T_{a_0}T_{a_5}T_{a_6}T_{a_7})(a_{2g+2}) = a_0.
\]
This product of twists is an element of $\Gamma$, showing that $T_{a_{2g+2}}$ is conjugate to $T_{a_0}$ by an element of $\Gamma$, and hence $T_{a_{2g+2}} \in \Gamma$ itself. 
\end{proof}

\begin{figure}[ht]
\labellist
\small
\pinlabel $T_{a_0}T_{a_5}T_{a_6}T_{a_7}$ at 220 225
\pinlabel $T_{a_7}T_{a_6}T_{a_5}T_{a_4}$ at 220 190
\pinlabel $T_{a_6}T_{a_5}T_{a_4}T_{a_3}$ at 165 145
\pinlabel $T_{a_5}T_{a_4}T_{a_3}T_{a_2}$ at 220 75
\pinlabel $a_{2g+2}$ at 150 280
\pinlabel $a_{0}$ at 165 40
\endlabellist
\includegraphics{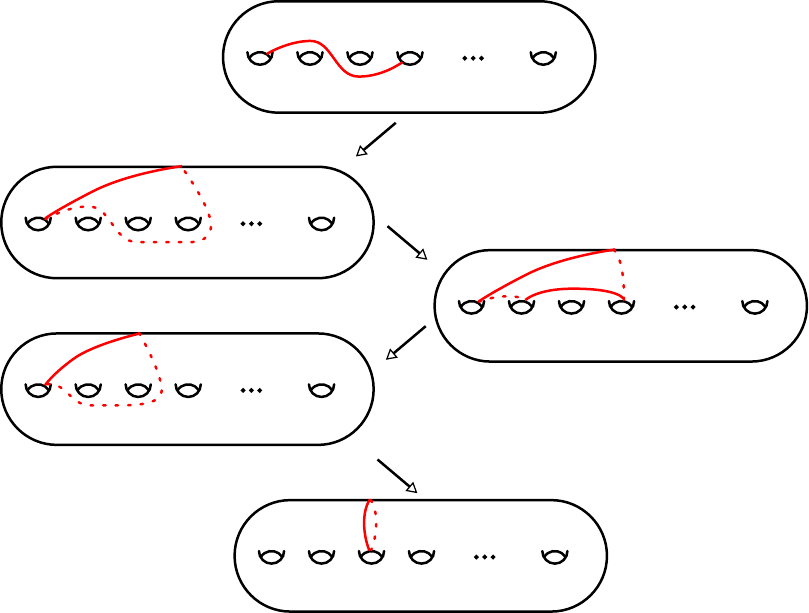}
\caption{The sequence of twists used to take $a_{2g+2}$ to $a_0$ in Lemma \ref{lemma:acurve}. All curves $a_i$ are labeled as in Figures \ref{figure:a} and \ref{figure:hardodd2}.}
\label{figure:acomp}
\end{figure}

\begin{lemma}\label{lemma:hardodd1}
The admissible twists $T_{a_{2g}}$ and $T_{a_{2g+1}}$ shown in Figure \ref{figure:hardodd3} are both contained in $\Gamma$. 
\end{lemma}
\begin{proof}
The curve $a_{2g}$ is obtained from $a_{2g+2}$ by sliding (see Example \ref{example:chainslide}). To find a sequence of twists about elements $T_{a_i} (0 \le i \le 2g-1)$ taking $a_{2g}$ to $a_{2g+1}$, we observe that the sequence of slides
\[(a_4, \ldots, a_8) \sim (a_5, \ldots, a_9) \sim (a_0, a_5, \ldots, a_8)\]
takes the curve $a_{2g}$ to $a_{2g+1}$.
\end{proof}

\begin{lemma}\label{lemma:hardodd}	
Let $b$ be the curve indicated in Figure \ref{figure:hardodd2}. Then $T_b^{2g-2} \in \Gamma$ in case \ref{case:hardodd} of Theorem \ref{theorem:genset}.
\end{lemma}
\begin{proof}
\begin{figure}
\labellist
\small
\pinlabel $a_0$ [r] at 167.2 83.2
\pinlabel $a_1$ [b] at 52 66.4
\pinlabel $a_2$ [b] at 81.6 60.8
\pinlabel $a_4$ [b] at 145 60.8
\pinlabel $a_{2g-1}$ [b] at 308.8 65.6
\pinlabel $b'$ [r] at 108 83.2
\pinlabel $c$ at 255 83.2
\pinlabel $c'$ at 102 20
\pinlabel $b$ at 224 83.2
\endlabellist
\includegraphics{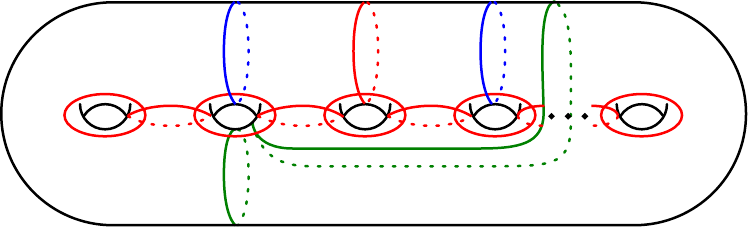}
\caption{The configuration of curves defining $\Gamma$ in case \ref{case:hardodd}, along with the auxiliary curves $b',c,c'$.}
\label{figure:hardodd2}
\end{figure}

As in Case \ref{case:hardeven}, we will first show $T_{b'}^{2g-2} \in \Gamma$ for a different curve $b'$, and subsequently show that $b, b'$ are in the same $\Gamma$--orbit. Consider first the $\mathscr D_{5}$--configuration determined by $a_0,a_4,a_5,a_6,a_7$ with boundary components $b',c$. By the $D$ relation (Lemma \ref{lemma:dn})
\[
T_{b'}^3 T_c \in \Gamma.
\]
Next consider the $\mathscr D_{2g-3}$--configuration determined by $a_0, a_4, a_5, \dots, a_{2g-1}$. By the $D$ relation,
\[
T_{b'}^{2g-5} T_{c'} \in \Gamma.
\]
As in Lemma \ref{lemma:hardeven}, it now suffices to show that $T_c T_{c'} \in \Gamma$. To see this, observe that the sequence $(a_0, a_5,a_6,a_7,a_{2g+2},a_1,a_2)$ forms a $7$--chain with boundary components $c,c'$. Thus
\[
T_c T_{c'} \in \Gamma
\]
by an application of the chain relation. Combining these elements yields $T_{b'}^{2g-2} \in \Gamma$; the $\Gamma$--equivalence of $b$ and $b'$  follows by the sliding principle (Lemma \ref{lemma:sliding}). 
\end{proof}

We are now in a position to complete the proof of Proposition \ref{lemma:gammaadmiss} by applying the machinery developed in Sections \ref{subsection:subspin} and \ref{subsection:networks}.

\begin{proof}[Proof of Proposition \ref{lemma:gammaadmiss}]
Recall that we are proving that the group $\Gamma$ generated by the collections of Dehn twists in cases \ref{case:hardeven} and \ref{case:hardodd} of Theorem \ref{theorem:genset} is equal to the admissible subgroup $\mathcal T_\phi$. One containment is vacuous, and so it remains to show that $\mathcal T_\phi \le \Gamma$.

We begin by showing that the generating sets of cases \ref{case:hardeven} and \ref{case:hardodd} can be completed to suitable arboreal networks; applying Lemma \ref{lemma:networkgens} then implies that $\Gamma$ contains a spin subsurface push subgroup $\widetilde \Pi(b)$. Lemma \ref{lemma:pushmakesT} then proves that $\mathcal T_\phi \le \Gamma$. As the networks are different in each case, we split our proof accordingly.

\para{Case \ref{case:hardeven}}
Consider the configuration of curves of $\Sigma_g$ as labeled in Figure \ref{figure:hardeven2}; by Lemma \ref{lemma:hardeven}, we have that $T_b^{2g-2} \in \Gamma$. Now observe that $a_0 \cup a_4 \cup b$ forms a pair of pants and $\{T_{a_i} \mid 0 \le i \le 2g-1,\ i \ne 3\}$ forms an arboreal filling network on $\overline{S'}$, the surface obtained by cutting along $b$ and capping off the left side with a disk. Therefore by Lemma \ref{lemma:networkgens} we know that $\Gamma$ contains $\widetilde \Pi(b)$. Since the $3$--chain $(a_4,a_3,b)$ satisfies the hypotheses of the $3$--chain $(a,a',b)$ of Lemma \ref{lemma:pushmakesT}, we can conclude that $\Gamma$ contains the admissible subgroup $\mathcal T_\phi$.

\para{Case \ref{case:hardodd}}
We now consider the configuration of curves as labeled in Figure \ref{figure:hardodd3}. By Lemmas \ref{lemma:hardodd1} and \ref{lemma:hardodd}, we know that $T_{a_{2g}}$, $T_{a_{2g+1}}$, and $T_b^{2g-2}$ are all contained in $\Gamma$. The network $\mathcal N$ defined in Figure \ref{figure:hardodd3} is evidently arboreal and fills $\overline{S'}$, and $a_8 \cup a_{2g+1} \cup b$ forms a pair of pants, so we may invoke Lemma \ref{lemma:networkgens} to conclude that $\widetilde \Pi(b) \leqslant \Gamma$. Applying Lemma \ref{lemma:pushmakesT} to the chain $(a_8,a_7,b)$ therefore implies that $\Gamma$ contains $\mathcal T_\phi$, concluding the proof of Proposition \ref{lemma:gammaadmiss}.
\end{proof}

\section{Spin structure stabilizers and the admissible subgroup}\label{section:admiss=stab}

The second step in the proof of Theorem \ref{theorem:genset} is to show that the admissible subgroup $\mathcal T_\phi$ coincides with the full spin structure stabilizer $\Mod_g[\phi]$. This is the counterpart to \cite[Propositions 5.1 and 6.2]{Salter_monodromy}. Those results only applied for $g$ sufficiently large\footnote{There is a typo in the statement of \cite[Proposition 5.1]{Salter_monodromy} -- the range should be $g \ge 5$, not $g \ge 3$ as claimed.} and imposed the requirement $r < g-1$. Moreover, in the case of $r$ even, \cite[Proposition 6.2]{Salter_monodromy} does not assert the equality between the admissible twist group $\mathcal T_\phi$ and the $r$-spin mapping class group $\Mod_g[\phi]$, only that $[\Mod_g[\phi] : \mathcal T_\phi]< \infty$. Proposition \ref{prop:admissfull} deals with all of these issues at once.
	\begin{proposition} \label{prop:admissfull}
	Let $\phi$ be an $r$--spin structure on a surface $\Sigma_g$ of genus $g \ge 3$. Then $\mathcal T_\phi = \Mod_g[\phi]$.
	\end{proposition}
When $r  = 2g-2$, this will complete the proof of Theorem \ref{theorem:genset}\ref{case:hardeven} and \ref{theorem:genset}\ref{case:hardodd}. The proof of Theorem \ref{theorem:genset}\ref{case:r} follows quickly, and is contained in Section \ref{section:generalr}.

	The proof of the Proposition is again accomplished in stages. In Section \ref{subsection:firststep} we outline the strategy and establish the first of three substeps. In Section \ref{section:ccp}, we discuss various versions of the ``change--of--coordinates principle'' in the presence of an $r$--spin structure. In the following Sections \ref{subsection:johnson} and \ref{subsection:torelli} we use these results to carry out the second and third substeps, respectively. 
	
\subsection{Outline: the Johnson filtration}\label{subsection:firststep} Once again the outline follows that given in \cite{Salter_monodromy} -- compare to Sections 5 and 6 therein. For any $r$--spin structure $\phi$, there is the evident containment 
\[
\mathcal T_\phi \le \Mod_g[\phi].
\]
To obtain the opposite containment, we appeal to the {\em Johnson filtration} of $\Mod(\Sigma_g)$. For our purposes, we need only consider the three--step filtration
\[
\mathcal K_g \le \mathcal I_g \le \Mod(\Sigma_g).
\]
The subgroup $\mathcal I_g$ is the {\em Torelli group}. It is defined as the kernel of the symplectic representation $\Psi: \Mod(\Sigma_g) \to \Sp(2g,\Z)$ which sends a mapping class $f$ to its induced action $f_*$ on $H_1(\Sigma_g; \Z)$. Set 
\[
H := H_1(\Sigma_g; \Z).
\] 
The group $\mathcal K_g$ is the {\em Johnson kernel}. It is defined as the kernel of the {\em Johnson homomorphism} (see Lemma \ref{lemma:johnsonhom})
\[
\tau: \mathcal I_g \to \wedge^3 H/ H.
\]
There is an alternate characterization of $\mathcal K_g$ due to Johnson. 
\begin{theorem}[Johnson \cite{Johnson_kernel}]\label{theorem:johnson2}
Let $\mathcal C$ denote the set of separating curves $c \subset \Sigma_g$ where $c$ bounds a subsurface of genus at most $2$. For $g \ge 3$, there is an equality
\[
\mathcal K_g = \pair{T_c \mid c \in \mathcal C}.
\]
\end{theorem}

The containment of the spin structure stabilizer $\Mod_g[\phi]$ inside of the admissible subgroup $\mathcal T_\phi$ will follow from a sequence of three lemmas. In preparation for Lemma \ref{step1}, recall from Section \ref{subsection:arf} that an $r$--spin structure for $r$ even determines an associated quadratic form (Remark \ref{remark:SSandQF}), as well as the algebraic stabilizer subgroup $\Sp(2g,\Z)[q]$ of Definition \ref{definition:algstab}. 

Recall that $\Psi:\Mod_g \rightarrow \Sp(2g, \Z)$ denotes the classical symplectic representation.
\begin{lemma}[Step 1]\label{step1} Fix $g \ge 3$ and let $\phi$ be an $r$--spin structure on $\Sigma_g$. If $r$ is odd, there is an equality
\[
\Psi(\Mod_g[\phi]) = \Psi(\mathcal T_\phi) = \Sp(2g,\Z).
\]
If $r$ is even, let $q$ denote the quadratic form on $H_1(\Sigma_g, \Z/2\Z)$ associated to $\phi$. Then there is an equality 
\[
\Psi(\Mod_g[\phi]) = \Psi(\mathcal T_\phi) = \Sp(2g,\Z)[q].
\]
\end{lemma}
\begin{lemma}[Step 2]\label{step2}
For $g \ge 3$, both $\Mod_g[\phi]$ and $\mathcal T_\phi$ contain the Johnson kernel $\mathcal K_g$.
\end{lemma}
\begin{lemma}[Step 3]\label{step3}
For $g \ge 3$ there is an equality
\[\tau(\Mod_g[\phi] \cap \mathcal I_g) = \tau(\mathcal T_\phi \cap \mathcal I_g)\]
of subgroups of $\wedge^3 H / H$.
\end{lemma}

Lemma \ref{step1} was established as \cite[Lemmas 5.4 and 6.4]{Salter_monodromy}. Lemma \ref{step2} is established in Section \ref{subsection:johnson}. The proof of Lemma \ref{step3} relies on the previous two steps, and is established in Section \ref{subsection:torelli}.\\

\subsection{Change--of--coordinates}\label{section:ccp}
The classical change--of--coordinates principle (c.f. \cite[Section 1.3]{FarbMarg}) describes the orbits of various configurations of curves and subsurfaces under the action of the mapping class group. When the underlying surface is equipped with an $r$--spin structure $\phi$, we will need to understand $\Mod_g[\phi]$--orbits of configurations as well. The results below (Lemma \ref{lemma:ccpodd}--\ref{lemma:ccphomol}) all present various facets of the change--of--coordinates principle in the presence of a spin structure. We will not prove these statements; Lemmas \ref{lemma:ccpodd} and \ref{lemma:ccpeven} are taken from \cite[Section 4]{Salter_monodromy} verbatim, while Lemmas \ref{lemma:ccpcurves} and \ref{lemma:ccphomol} follow easily from the techniques therein. All of these results are ultimately a consequence of the classification of orbits of $r$-spin structures under $\Mod(\Sigma_g)$, originally obtained by Randal--Williams \cite[Theorem 2.9]{RW}. 

\begin{lemma}\label{lemma:ccpodd}
Let $r$ be an odd integer, and let $\Sigma_g$ be a surface of genus $g \ge 2$ equipped with an $r$--spin structure $\phi$. Let $S \subset \Sigma_g$ be a (not necessarily proper) subsurface of genus $h \ge 2$ with a single boundary component. Then the following assertions hold:
\begin{enumerate}
\item\label{item:oddgsb} For any $2h$--tuple $(i_1, j_1, \dots, i_h, j_h)$ of elements of $\Z / r \Z$, there is some geometric symplectic basis $\mathcal B = \{a_1, b_1, \dots, a_h, b_h\}$ for $S$ with $\phi(a_\ell) = i_\ell$ and $\phi(b_\ell) = j_\ell$ for all $1 \le \ell \le h$,
\item\label{item:oddchain} For any $2h$--tuple $(k_1, \dots, k_{2h})$ of elements of $\Z/r\Z$, there is some chain $(a_1, \dots, a_{2h})$ of curves on $S$ such that $\phi(a_\ell) = k_\ell$ for all $1 \le \ell \le 2h$.
\end{enumerate}
\end{lemma}

\begin{lemma}\label{lemma:ccpeven}
Let $r$ be an even integer, and let $\Sigma_g$ be a surface of genus $g \ge 2$ equipped with an $r$--spin structure $\phi$. Let $S \subset \Sigma_g$ be a (not necessarily proper) subsurface of genus $h \ge 2$ with a single boundary component. Then the following assertions hold:
\begin{enumerate}
\item\label{item:gsbeven} For a given $2h$--tuple $(i_1, j_1, \dots, i_h, j_h)$ of elements of $\Z / r \Z$, there is some geometric symplectic basis $\mathcal B = \{a_1, b_1, \dots, a_h, b_h\}$ for $S$ with $\phi(a_\ell) = i_\ell$ and $\phi(b_\ell) = j_\ell$ for $1 \le \ell \le h$ if and only if the parity of the spin structure defined by these conditions agrees with the parity of the restriction $\phi |_{S}$ to $S$. 
\item\label{item:gsbevenrestricted} For {\em any} $(2h-2)$--tuple $(i_1, j_1, \dots,i_{h-1}, j_{h-1})$ of elements of $\Z / r \Z$, there is some geometric symplectic basis $\mathcal B = \{a_1, b_1, \dots, a_h, b_h\}$ for $S$ with $\phi(a_\ell) = i_\ell$ and $\phi(b_\ell) = j_\ell$ for $1 \le \ell \le h-1$.
\item \label{item:chaineven} For a given $2h$--tuple $(k_1, \dots, k_{2h})$ of elements of $\Z/r\Z$, there is some chain $(a_1, \dots, a_{2h})$ of curves on $S$ such that $\phi(a_\ell) = k_\ell$ for all $1 \le \ell \le 2h$ if and only if the parity of the spin structure defined by these conditions agrees with the parity of the restriction $\phi |_S$ to $S$. 
\item \label{item:chainevenrestricted} For {\em any} $(2h-2)$--tuple $(k_1, \dots, k_{2h-2})$ of elements of $\Z/r\Z$, there is some chain $(a_1, \dots, a_{2h-2})$ of curves on $S$ such that $\phi(a_\ell) = k_\ell$ for all $1 \le \ell \le 2h-2$.
\end{enumerate}
\end{lemma}

One of the most important iterations of the change--of--coordinates principle is that the spin stabilizer subgroup acts transitively on the set of all curves with a given winding number.

\begin{lemma}\label{lemma:ccpcurves}
Let $\phi$ be a $r$--spin structure, and let $c,d \subset \Sigma_g$ be nonseparating curves. If $\phi(c) = \phi(d)$, then there is some $f \in \Mod_g[\phi]$ such that $f(c) = d$.
\end{lemma}

One can also use this principle to find curves in a given homology class with given winding number (subject to Arf invariant restrictions, when applicable).

\begin{lemma}\label{lemma:ccphomol}
Let $\phi$ be an $r$--spin structure on $\Sigma_g$ and let $z \in H_1(\Sigma_g;\Z)$ be fixed. If $r$ is odd, then for any element $k \in \Z/r\Z$, there is a simple closed curve $c$ satisfying $\phi(c) = k$ and $[c] = z$. If $r$ is even, let $\epsilon \in \Z/2\Z$ denote the mod $2$ value of $\phi(z)$ in the sense of Lemma \ref{lemma:2spinhom}. Then there exists a simple closed curve $c$ satisfying $\phi(c) = k$ and $[c] = z$ if and only if $k \equiv \epsilon \pmod 2$. 
\end{lemma}

\subsection{Step 2: Containment of the Johnson kernel}\label{subsection:johnson}
Our objective in this section is to establish Lemma \ref{step2}, showing that the admissible subgroup $\mathcal T_\phi$ contains the Johnson kernel $\mathcal K_g$. We establish one simple preliminary lemma first.
\begin{lemma}\label{lemma:hyperellipticparity}
Let $\Sigma_{g,1}$ be a surface with one boundary component, and suppose $a_1, \dots, a_{2g}$ is a maximal chain. Suppose that $\phi$ is a $2g-2$-spin structure on $\Sigma_{g,1}$ for which each $a_i$ is admissible. Then the parity of $\phi$ is given as follows:
\[
\Arf(\phi) = \begin{cases} 	1& g \equiv 1,2 \pmod 4\\
					0& g \equiv 3,0 \pmod 4.

\end{cases}
\]
Conversely, if $\Sigma_{g,1}$ is equipped with a $(2g-2)$-spin structure $\phi$ with $\Arf(\phi)$ as above, then there exists a maximal chain $a_1, \dots, a_{2g}$ of admissible curves.

\end{lemma}
\begin{proof}
The first assertion is an easy computation following from the definition of the Arf invariant (Definition \ref{definition:arfQF}) and homological coherence (Lemma \ref{lemma:homcoh}) which we omit. The second claim follows directly from Lemma \ref{lemma:ccpeven}.\ref{item:chaineven}.
\end{proof}

\begin{proof}[Proof of Lemma \ref{step2}]
Suppose that $\phi$ is an $r$--spin structure, and $\tilde \phi$ is a $({2g-2})$--spin structure which refines $\phi$. Then $\T_{\tilde \phi} \le \T_\phi$, and hence it suffices to prove the lemma in the case when $\phi$ is a $({2g-2})$--spin structure.

By Theorem \ref{theorem:johnson2}, it suffices to prove that if $c$ is a separating curve bounding a subsurface of genus at most $2$, then $T_c \in \T_\phi$.

\para{Genus 1}
We begin by considering the genus 1 case, so suppose that $a$ is a curve which bounds a genus 1 subsurface $W_a$.
Observe that if either
\begin{itemize}
\item $ g \equiv 2, 3 \pmod 4$ and $\Arf(\phi) = 1 + \Arf\left( \phi |_{W_a} \right)$ or
\item $ g \equiv 0, 1 \pmod 4$ and $\Arf(\phi) = \Arf\left( \phi |_{W_a} \right)$,
\end{itemize}
then the complementary subsurface $\Sigma_g \setminus W_a$ has
\begin{itemize}
\item $ g(\Sigma_g \setminus W_a) \equiv 1, 2 \pmod 4$ and $\Arf\left( \phi |_{\Sigma_g \setminus W_a} \right) = 1$ or
\item $ g(\Sigma_g \setminus W_a) \equiv 3, 0 \pmod 4$ and $\Arf\left( \phi |_{\Sigma_g \setminus W_a} \right) = 0$,
\end{itemize}
respectively. In either of the above cases, Lemma \ref{lemma:hyperellipticparity} asserts that there exists a maximal chain of admissible curves on $\Sigma_g \setminus W_a$, and hence by the chain relation (Lemma \ref{lemma:chain}), $T_a \in \T_\phi$.

Suppose now that we are not in one of the cases above, so the complementary subsurface $\Sigma_g \setminus W_a$ does not admit a maximal chain of admissible curves. In order to exhibit the twist on $a$, we will form a lantern relation and prove that the other terms in the relation lie in $\T_\phi$.

By Lemma \ref{lemma:ccpeven}.\ref{item:chainevenrestricted}, there exists a $2g-4$-chain of admissible curves on $\Sigma_g \setminus W_a$, filling a subsurface of genus $g-2$. Let $b$ be a separating curve that bounds a subsurface $W_b$ of genus 1 disjoint from this chain. The chain relation (Lemma \ref{lemma:chain}) then implies that $T_b \in \T_\phi$.

Let $c$ be any curve in $\Sigma_g \setminus (W_a \cup W_b)$ such that $\phi(c) = -2$, and take $d$ to be a curve which together with $a, b,$ and $c$ bounds a four--holed sphere (in the case $g = 3$, necessarily $c = d$, but this is not a problem). By homological coherence (Lemma \ref{lemma:homcoh}), $\phi(d) = -2$.
These curves fit into a lantern relation as shown in Figure \ref{fig:lantern1}.

\begin{figure}[ht]
\labellist
\small
\pinlabel $W_a$ [bl] at 164 168.8
\pinlabel $W_b$ [br] at 16 108
\pinlabel $x$ [bl] at 152 56.8
\pinlabel $y$ [tl] at 152 95.2
\pinlabel $z$ [l] at 80 75.2
\pinlabel $a$ [bl] at 168 135.2
\pinlabel $b$ [tr] at 60 28.8
\pinlabel $c$ [l] at 170 8
\pinlabel $d$ [bl] at 182 120
\endlabellist
\includegraphics{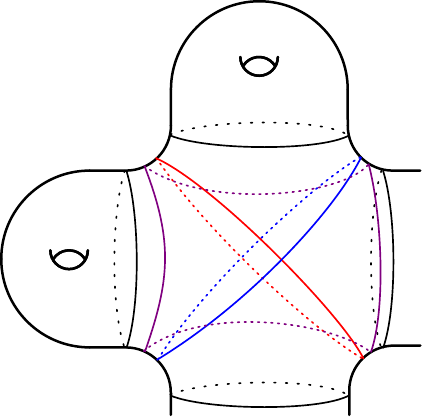}
\caption{The lantern for genus 1.}
\label{fig:lantern1}
\end{figure}

By construction, $c \cup d$ bounds a subsurface $V \cong \Sigma_2^2$. Choose a subsurface $V' \subset V$ with a single boundary component $y$ such that $V'$ contains both $W_a$ and $W_b$. Since the Arf invariants of $W_a$ and $W_b$ are of opposite parity, we have
\[\Arf\left( \phi |_ {V'} \right) = 1.\]
By the change of coordinates principle (Lemma \ref{lemma:ccpeven}.\ref{item:chaineven}) there exists a maximal chain $\{a_1, a_2, a_3, a_4\}$ of admissible curves on $V'$; by the chain relation, $T_y \in \T_\phi$.
Now let $a_5 \subset V$ be a curve disjoint from $a_2$ so that $i(a_4, a_5)=1$ and which together with $a_1$, $a_3$, and $c$ bounds a four--holed sphere. By homological coherence (Lemma \ref{lemma:homcoh}), we have that $\phi(a_5)=0$. Therefore, $\{a_1, \ldots, a_5\}$ is a maximal chain of admissible curves on $V$ and so by the chain relation, we have that $T_c T_d \in \T_\phi$.

Finally, we note that the pairs $(c,x)$ and $(d, z)$ each bound subsurfaces of genus $1$ with two boundary components. Since $\phi(c) = \phi(d) = -2$, homological coherence implies that $x$ and $z$ must both be admissible.

Applying the lantern relation (Lemma \ref{lemma:lantern}), we have that
\[T_a = (T_bT_cT_d)^{-1} (T_x T_y T_z) \in \T_\phi.\]
Observe that if $g=3$, then this is enough to finish the proof, since every separating twist is of genus $1$.

\para{Genus 2}
Now suppose that $g \ge 4$ and let $x$ be a curve bounding a subsurface $W_x$ of genus $2$ (this choice of label will allow Figures \ref{fig:lantern1} and \ref{fig:lantern2} to share a labeling system). Observe that if the Arf invariant of $\phi |_{W_x}$ is odd, then by the change--of--coordinates principle (Lemma \ref{lemma:ccpeven}.\ref{item:chaineven}), $W_x$ admits a maximal chain of admissible curves and so by applying the chain relation, $T_x \in \T_\phi$.

So suppose that $\Arf\left( \phi |_{W_x} \right) =0$. By the change--of--coordinates principle (in particular, Lemma \ref{lemma:ccpeven}.\ref{item:gsbeven}), we can choose two disjoint subsurfaces $W_b \subset W_x$ and $W_c \subset W_x$ each homeomorphic to $\Sigma_1^1$ such that
\[ 
\Arf\left( \phi |_{W_b} \right) = \Arf\left( \phi |_{W_c} \right) = 1.
\]
Let their corresponding boundaries be $b$ and $c$.
Again appealing to Lemma \ref{lemma:ccpeven}.\ref{item:gsbeven}, choose $a$ to be a curve bounding a subsurface $W_a$ of $\Sigma_g \setminus W_x$ with $W_ a \cong \Sigma_1^1$ such that
\[
\Arf\left( \phi |_{W_a} \right) = 0.
\]
By the genus--1 case established above, we know that $T_a , T_b, T_c \in \T_\phi$. Finally, choose $d$ to be any curve in $\Sigma_g \setminus (W_x \cup W_a)$ which bounds a pair of pants together with $a$ and $x$. The curves then fit into a lantern relation as in Figure \ref{fig:lantern2}.

\begin{figure}[ht]
\labellist
\small
\pinlabel $W_a$ [bl] at 160 237.6
\pinlabel $W_b$ [br] at 12.8 165.6
\pinlabel $x$ [bl] at 155.2 108.8
\pinlabel $y$ [tl] at 155.2 140
\pinlabel $z$ [l] at 80.8 131.2
\pinlabel $a$ [bl] at 169.6 188.8
\pinlabel $b$ [tr] at 60 80.8
\pinlabel $c$ [l] at 171.2 59.2
\pinlabel $d$ [bl] at 187.2 172.8
\pinlabel $W_c$ [tl] at 160.8 17.6
\endlabellist
\includegraphics{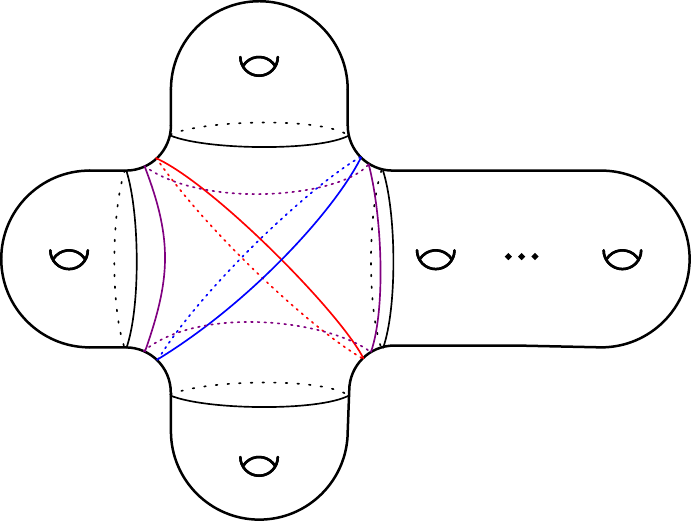}
\caption{The lantern for genus 2 separating twists.}
\label{fig:lantern2}
\end{figure}

Let $W_d$ denote the subsurface bounded by $d$ which contains $W_a, W_b,W_c$. By construction, we have $\Arf(\phi |_{W_d}) = 0$. By the change of coordinates principle (Lemma \ref{lemma:ccpeven}.\ref{item:chaineven}), it follows that $W_d$ admits a maximal chain of admissible curves, and so $T_d \in \T_\phi$ by the chain relation.

Finally, observe that the curves $y$ and $z$ shown in Figure \ref{fig:lantern2} bound subsurfaces of genus $2$ with odd Arf invariant, and so both admit maximal chains of admissible curves. Thus $T_y, T_z \in \T_\phi$.

Again applying the lantern relation (Lemma \ref{lemma:lantern}), we see
\[T_x = (T_a T_b T_c T_d) (T_y T_z)^{-1} \in \T_\phi.\]
Therefore, since the separating twists of genus one and two generate the Johnson kernel (Theorem \ref{theorem:johnson2}), we see that $\K_g < \T_\phi$.
\end{proof}

\subsection{Step 3: Intersection with the Torelli group}\label{subsection:torelli}
In order to show the equality $\tau(\Mod_g[\phi] \cap \mathcal I_g) = \tau(\mathcal T_\phi \cap \mathcal I_g)$, it is necessary to give a precise description of the subgroup $\tau(\Mod_g[\phi] \cap \mathcal I_g)$. We begin with a brief summary of the theory of the Johnson homomorphism $\tau$. It is not necessary to know a construction; we content ourselves with a minimal account of its properties.

\para{The Johnson homomorphism} Recall the notation $H := H_1(\Sigma_g; \Z)$, and observe that there is an embedding
\[
H\into \wedge^3 H
\]
defined by $x \mapsto x \wedge \omega$, with $\omega = x_1 \wedge y_1 + \dots + x_g \wedge y_g$ for some symplectic basis $x_1, y_1, \dots, x_g, y_g$.  

\begin{lemma}[Johnson, \cite{johnsonhom}]\label{lemma:johnsonhom}
\ 
\begin{enumerate}
\item There is a surjective homomorphism $\tau: \mathcal I_g \to \wedge^3 H/ H$ known as the {\em Johnson homomorphism}; as the target is abelian, $\tau$ factors through $H_1(\mathcal I_g;\Z)$. It is $\Sp(2g,\Z)$--equivariant with respect to the action of $\Sp(2g,\Z)$ on $H_1(\mathcal I_g;\Z)$ induced by conjugation by $\Mod(\Sigma_g)$ and the evident action of $\Sp(2g, \Z)$ on $\wedge^3 H / H$. 
\item \label{item:bpformula} Let $c\cup d$ bound a subsurface $\Sigma_{h,2}$. Choose any further subsurface $\Sigma_{h,1} \subset \Sigma_{h,2}$, and let $\{x_1, y_1, \dots, x_h, y_h\}$ be a symplectic basis for $H_1(\Sigma_{h,1};\Z)$. Then 
\[
\tau(T_c T_d^{-1}) = (x_1 \wedge y_1 + \dots + x_h \wedge y_h)\wedge [c],
\]
where $c$ is oriented with $\Sigma_{h,2}$ to the left. In the case $h = 1$, if $\alpha,\beta,\gamma$ is a maximal chain on $\Sigma_{1,2}$, then 
\[
\tau(T_cT_d^{-1}) = [\alpha] \wedge [\beta] \wedge [\gamma].
\]

\end{enumerate}
\end{lemma}

To describe $\tau(\Mod_g[\phi] \cap \mathcal I_g)$, we consider a contraction of $\wedge^3 H/ H$. Lemma \ref{lemma:contraction} is well known; see, e.g. \cite[Sections 5,6]{johnsonhom}.
\begin{lemma}\label{lemma:contraction}
For any $s$ dividing $g-1$, there is an $\Sp(2g,\Z)$--equivariant surjection
\[
C_s: \wedge^3 H / H \to H_1(\Sigma_g; {\Z/s\Z})
\]
given by the contraction
\[
C(x\wedge y \wedge z) = \pair{x,y}z+ \pair{y,z} x+ \pair{z,x}y \pmod{s}.
\]
\end{lemma}

Although it was not formulated in this language, Johnson showed that the contraction $C$ vanishes on the group $\tau(\Mod_g[\phi] \cap \mathcal I_g)$.

\begin{lemma}\label{lemma:characterizetorelli}
Let $\phi$ be an $r$--spin structure on a surface $\Sigma_g$ of genus $g \ge 3$. Set $s = r$ if $r$ is odd, and $s = r/2$ if $r$ is even. Then $C_s \circ \tau = 0$ on $\Mod_g[\phi] \cap \mathcal I_g$. 
\end{lemma}
\begin{proof}
We recall (c.f. \cite{Chill_wn1}, see also \cite[Section 6]{johnsonhom} and \cite[Theorem 5.5]{Salter_monodromy}) that the ``mod--$r$ Chillingworth invariant'' is a homomorphism 
\[
c_r: \mathcal I_g \to 2H_1(\Sigma_g; \Z/r\Z) \cong H_1(\Sigma_g; \Z/s\Z)
\]
with the property that $c_r(f) = 0$ for $f \in \mathcal I_g$ if and only if $f$ preserves {\em all} $r$--spin structures. For $r'$ dividing $r$, the invariants $c_r$ and $c_{r'}$ are compatible in the sense that $c_{r'} = c_r \pmod{r'}$.

If $r$ is odd, then there is a natural identification of the kernels of $c_r$ and of $c_{2r}$, for
\[2H_1(\Sigma_g; \Z/2r\Z) \cong H_1(\Sigma_g; \Z/r\Z) \cong 2H_1(\Sigma_g; \Z/r\Z).\]
Thus it suffices to consider the case of $r$ even.

According to \cite[Theorem 3]{johnsonhom}, there is an equality
\[
C_{g-1} \circ \tau = c_{2g-2}.
\]
This establishes the claim in the case $r = 2g-2$. The general case now follows by reduction mod $r$. 
\end{proof}

We will show that the constraint of Lemma \ref{lemma:characterizetorelli} in fact {\em characterizes} the groups $\tau(\T_\phi \cap \I_g)$ and $\tau(\Mod_g[\phi] \cap \I_g)$. Lemma \ref{lemma:refine2} refines the statement of Lemma \ref{step2}; our goal in the remainder of the subsection is to prove Lemma \ref{lemma:refine2} and so accomplish Step 2.

\begin{lemma}\label{lemma:refine2}
Set $s$ as in Lemma \ref{lemma:characterizetorelli}. Then there is an equality $\tau(\T_\phi \cap \I_g) = \ker(C_s)$. Consequently, $\tau(\Mod_g[\phi] \cap \mathcal I_g) = \tau(\mathcal T_\phi \cap \mathcal I_g)$.
\end{lemma}

This will follow by first exhibiting a generating set for the kernel of the contraction $C_s$ (Lemma \ref{lemma:spgenset}) and then finding elements of $\T_\phi \cap \I_g$ realizing these elements (Lemma \ref{lemma:Texhibit}). 

\para{Symplectic linear algebra} To find the generators for $\ker(C_s)$ in $\tau(\T_\phi \cap \I_g)$, we will make heavy use of the $\Sp(2g,\Z)$--equivariance of $\tau$ asserted in Lemma \ref{lemma:johnsonhom}.1. We begin with some results in symplectic linear algebra to this end. We will only need the result of Lemma \ref{lemma:symptools} in the proof; the Lemmas \ref{lemma:existsA} and \ref{lemma:extendB} are preliminary. 

Let $H$ be a free $\Z$--module of rank $2g \ge 6$ equipped with a symplectic form $\pair{\cdot,\cdot}$, and suppose that $q$ is a nondegenerate quadratic form on $H \otimes (\Z/2Z) \cong (\Z / 2\Z)^{2g}$ (see Definition \ref{definition:QF}). Given such a $q$, the {\em $q$--vector} of a symplectic basis $\mathcal B = (x_1, y_1, \dots, x_g,y_g)$ for $H$ is the element $\vec q(\mathcal B)$ of $(\Z/2\Z)^{2g}$ given by $\vec q(\mathcal B) = (q([x_1]), \dots, q([y_g]))$. Recall also the definition of the algebraic stabilizer $\Sp(2g,\Z)[q]$ of a mod-$2$ quadratic form discussed in Definition \ref{definition:algstab}.
	\begin{lemma}\label{lemma:existsA}
	If $\mathcal B = (x_1, y_1, \dots x_g, y_g)$ and $\mathcal B' = (x_1',y_1', \dots, x_g', y_g')$ are symplectic bases with $\vec q(\mathcal B) = \vec q(\mathcal B')$, then there is $A \in \Sp(2g,\Z)[q]$ such that $A(\mathcal B) = \mathcal B'$.
	\end{lemma}
	\begin{proof}
	There is some element $A \in \Sp(2g,\Z)$ such that $A(\mathcal B) = \mathcal B'$. We claim that necessarily $A \in \Sp(2g,\Z)[q]$. Let $q'$ be the quadratic form $q' = A \cdot q$. We wish to show that $q' = q$. It suffices to show that $\vec{q'}(\mathcal B') = \vec{q}(\mathcal B')$. By construction,
	\[
	\vec{q'}(\mathcal B' ) = \vec q(A^{-1} \mathcal B') = \vec q(\mathcal B) = \vec q(\mathcal B');
	\]
	the last equality holding by hypothesis.
	\end{proof}

In the statement of Lemma \ref{lemma:extendB} below, a {\em partial symplectic basis} is a collection of vectors $\{v_1, \dots, v_k\}$ with 	$\pair{v_{2i-1}, v_{2i}} = 1$ for all $2i \le k$ and all other pairings zero. We do not assume that $k$ is even. 
	
	\begin{lemma}\label{lemma:extendB}
	Let $q$ be a quadratic form, $\mathcal B = (x_1, y_1, \dots, x_g,y_g)$ a symplectic basis, and $\vec q(\mathcal B)$ the associated $q$--vector. Suppose $\{v_1, \dots, v_k\}$ is a partial symplectic basis, and moreover that $q(v_{2i-1}) = q(x_i)$ and $q(v_{2i}) = q(y_i)$ for all $2i \le k$. Then $\{v_1, \dots, v_k\}$ admits an extension to a symplectic basis $\mathcal B'$ with $\vec q(\mathcal B) = \vec q(\mathcal B')$. 
	\end{lemma}
	\begin{proof}
	If $k$ is odd, choose an arbitrary element $v_{k+1}$ satisfying $\pair{v_k, v_{k+1}} = 1$ and $q(v_{k+1}) = q(y_{(k+1)/2})$; we proceed with the argument under the assumption that $k$ is even. Let $V$ denote the orthogonal complement to $\{x_1, y_1, \dots, x_{k/2}, y_{k/2}\}$; this is a symplectic $\Z$--module of rank $2g-k$ equipped with a quadratic form $q|_{V}$ induced by the restriction of $q$. Likewise, let $W$ denote the orthogonal complement to $\{v_1, \dots, v_k\}$; then $W$ is also a symplectic $\Z$--module of rank $2g-k$ equipped with a quadratic form $q|_{W}$. Since the $q$--values of $\{x_1, \dots, y_{k/2}\}$ and $\{v_1, \dots, v_k\}$ agree and the Arf invariant is additive under symplectic direct sum (Remark \ref{remark:arfadd}), we conclude that $\Arf(q|_{V}) = \Arf(q|_{W})$. Thus there is a symplectic isomorphism $f: V \to W$ that transports the form $q|_{V}$ to $q|_{W}$. The symplectic basis
	\[
	\mathcal B' = \{v_1, \dots, v_k, f(x_{k/2+1}), f(y_{k/2+1}), \dots, f(x_g), f(y_g)\}
	\]
	satisfies $\vec q(\mathcal B) = \vec q(\mathcal B')$ by construction.
	\end{proof}
	
	\begin{lemma}\label{lemma:symptools}\ 	
	\begin{enumerate}
	\item\label{item:contractible} Let $v_1, v_2, v_3$ and $v_1', v_2', v_3'$ be partial symplectic bases for $H$. If $q(v_i) = q(v_i')$ for $i = 1,2,3$, then there is some element $A \in \Sp(2g,\Z)[q]$ such that $A(v_1\wedge v_2 \wedge v_3) = v_1'\wedge v_2' \wedge v_3'$. 
	\item \label{item:threetuple} Let $v_1, v_2, v_3$ and $v'_1, v_2', v_3'$ be triples such that $\pair{v_i, v_j} = \pair{v_i', v_j'} = 0$ for all pairs of indices $1 \le i < j \le 3$. If $q(v_i) = q(v_i')$ for $i = 1,2,3$, then there is some element $A \in \Sp(2g,\Z)[q]$ such that $A(v_1 \wedge v_2 \wedge v_3) = v_1'\wedge v_2' \wedge v_3'$.
	\end{enumerate}
	\end{lemma}
	\begin{proof}
	For \eqref{item:contractible}, we extend $\{v_1,v_2,v_3\}$ and $\{v_1',v_2',v_3'\}$ to symplectic bases $\mathcal B, \mathcal B'$. By Lemma \ref{lemma:extendB}, we can furthermore assume that $\vec q(\mathcal B) = \vec q(\mathcal B')$. The required element $A \in \Sp(2g,\Z)[q]$ is now obtained by an appeal to Lemma \ref{lemma:existsA}. 
	
	The proof of \eqref{item:threetuple} is very similar: $\{v_1, v_2, v_3\}$ and $\{v_1', v_2', v_3'\}$ can again be extended to symplectic bases $\mathcal B, \mathcal B'$ with equal $q$--vectors and the result follows by Lemma \ref{lemma:existsA}. 
	\end{proof}

\para{Some topological computations} Along with symplectic linear algebra, we will also need to see that $\mathcal T_\phi$ contains an ample supply of certain specific mapping classes.
\begin{lemma}\label{lemma:twistpowers}
Let $\phi$ be an $r$--spin structure on a surface $\Sigma_g$ of genus $g \ge 3$ (if $g = 3$, assume $\Arf(\phi) = 1$). Let $b \subset \Sigma_g$ be a nonseparating simple closed curve satisfying $\phi(b) = -1$. Then $T_b^r \in \mathcal T_\phi$. 
\end{lemma} 	
\begin{proof}
This will require a patchwork of arguments depending on the specific values of $r$ and $g$. For $g \ge 5$ and $r < g-1$, this was established in \cite[Lemma 5.2]{Salter_monodromy}. We will treat the remaining cases as follows: (1) for $g \ge 3$ and $r = 2g-2$, (2) for $g \ge 4$ and $r = g-1$, (3) the remaining sporadic cases appearing for $g \le 4$.

(1): Lemmas \ref{lemma:hardeven} and \ref{lemma:hardodd} furnish a specific $b$ with $\phi(b) = -1$ and $T_b^{2g-2} \in \T_\phi$. By the change--of--coordinates principle (specifically Lemma \ref{lemma:ccpcurves}), given {\em any} nonseparating $b'$ satisfying $\phi(b') = -1$, one can find an element $f \in \Mod_g[\phi]$ such that $f(b') = b$; consequently the elements $T_b^{2g-2}$ and $T_{b'}^{2g-2}$ are conjugate elements of $\Mod_g[\phi]$. To conclude the argument, we observe that $\T_\phi$ is a {\em normal} subgroup of $\Mod_g[\phi]$, so that $T_{b'}^{2g-2} \in \T_\phi$ as desired.

(2): Assume now $g \ge 4$ and $r = g-1$. Let $b$ be an arbitrary nonseparating curve satisfying $\phi(b) = -1$. Our first task is to find a certain configuration of admissible curves well--adapted to $b$; the $D$ relation (Lemma \ref{lemma:dn}) then allows us to exhibit $T_b^r \in \T_\phi$. The configuration we construct is depicted in Figure \ref{figure:g-1}. 

By the change--of--coordinates principle (Lemma \ref{lemma:ccpodd} or \ref{lemma:ccpeven}), there exists a chain of admissible curves $a_2, \dots, a_{2g-3}$ disjoint from $b$. Let $a_1'$ be any curve satisfying $i(a_1', b) = i(a_1', a_2) = 1$ and $i(a_1', a_j) = 0$ for $j \ge 2$. For any $k \in \Z$, the curves $T_b^k(a_1')$ have these same intersection properties. Since $\phi(b) = -1$, twist linearity (Definition \ref{definition:spin}.\ref{item:TL}) implies that we can choose $a_1 = T_b^k(a_1')$ for suitable $k$ such that $a_1$ is admissible. Finally, let $a_0$ be a curve so that $a_0 \cup a_2 \cup b$ forms a pair of pants to the left of $b$, and such that $i(a_0, a_3) = 1$ and $i(a_0, a_j) = 0$ for all other $j$. By homological coherence (Lemma \ref{lemma:homcoh}), $a_0$ is also admissible. Finally, let $d$ be chosen so that $b \cup a_2 \cup a_4 \cup \dots \cup a_{2g-8} \cup d$ bounds a subsurface to the left of $b$ of genus $2$ and $g-2$ boundary components, such that $i(d, a_{2g-7}) = 1$ and $i(d,a_j) = 0$ for other $j$. Since $r = g-1$, homological coherence (Lemma \ref{lemma:homcoh}) implies that $d$ is admissible. 

Consider the $\mathscr D_{2g-3}$ configuration determined by the curves $a_0, a_2, \dots, a_{2g-3}$. By construction, one boundary component is $b$, and the other is the curve $c$ shown in Figure \ref{figure:g-1}. Applying the $D$ relation to this configuration shows that 
\begin{equation}\label{first}
T_b^{2g-5} T_c \in \T_\phi.
\end{equation}
Consider next the $\mathscr D_{2g-6}$ configuration determined by the curves $a_0, a_2, \dots a_{2g-7}, d$. This configuration has boundary components $b,c$, and the {\em separating} curve $c'$. Applying the $D$ relation shows
\begin{equation}\label{second}
T_b^{g-4} T_c T_{c'} \in \T_\phi;
\end{equation}
since $c'$ is separating, we invoke Lemma \ref{step2} to conclude that also $T_b^{g-4} T_c \in \T_\phi$. Combining \eqref{first} and \eqref{second} shows that $T_b^{g-1} \in \T_\phi$.

(3) The remaining cases are $(g,r) = (3,2)$ and $(4,2)$. Let $b$ be a nonseparating curve satisfying $\phi(b) = -1$, and choose an admissible curve $a_1$ disjoint from $b$. Let $a_2$ be chosen so that $a_1 \cup a_2 \cup b$ forms a pair of pants; by homological coherence (Lemma \ref{lemma:homcoh}), $a_2$ is also admissible. By the change--of--coordinates principle (Lemmas \ref{lemma:ccpodd} and \ref{lemma:ccpeven}) , it is easy to find an admissible curve $a_0$ with the following intersection properties:
\[
i(a_0, b) = 0, \qquad i(a_0,a_1) = i(a_0,a_2) = 1.
\]
Finally, choose $a_3$ so that the following conditions are satisfied: $i(a_3, a_0) = 1$ and $a_3$ is disjoint from all other curves under consideration, and $a_1 \cup a_3$ bounds a subsurface of genus $1$ containing $b$. By homological coherence, $a_3$ is admissible, and by construction, $(a_0, a_1, a_2, a_3)$ forms a $\mathscr D_4$ configuration. In the notation of Figure \ref{figure:dnrel}, the boundary component $\Delta_0$ is separating, and the curves $\Delta_1$ and $\Delta_1'$ are both isotopic to $b$. By the $D$ relation (Lemma \ref{lemma:dn}),
\[
T_{\Delta_0}T_{b}^2 \in \T_\phi.
\]
By Lemma \ref{step2}, since $\Delta_0$ is separating, it follows that $T_{b}^2 \in \T_\phi$ as required. 
\end{proof}	

\begin{figure}
\labellist
\small
\pinlabel $b$ [r] at 72.8 84
\pinlabel $a_1$ [br] at 61.6 61.6
\pinlabel $a_0$ [r] at 135.2 84
\pinlabel $a_2$ [b] at 112.8 61.6
\pinlabel $a_{2g-7}$ [b] at 214.4 66.4
\pinlabel $d$ [r] at 208 22.4
\pinlabel $a_{2g-3}$ [b] at 324.8 66.4
\pinlabel $c$ [r] at 72.8 22.4
\pinlabel $c'$ at 235 95
\endlabellist
\includegraphics{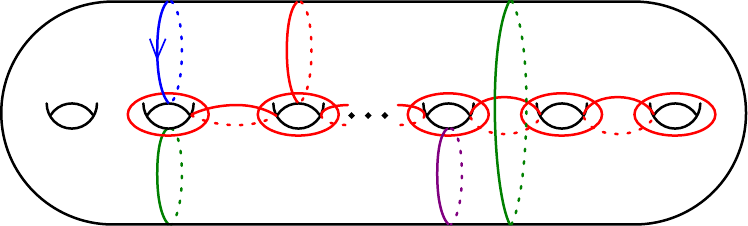}
\caption{The configuration used in the proof of Lemma \ref{lemma:twistpowers}.}
\label{figure:g-1}
\end{figure}

Taking connected sums of curves allows us to build new curves from old ones and relate the winding number of the new to that of the old. This construction will be used in Lemma \ref{lemma:ctwist} below.
\begin{definition}[Connected sums]\label{definition:CAS}
Let $a$ and $b$ be disjoint oriented simple closed curves, and let $\epsilon$ be an embedded arc connecting the left side of $a$ to that of $b$ so that $\epsilon$ is otherwise disjoint from $a \cup b$. A regular neighborhood of $ a \cup \epsilon \cup b$ is then a three--holed sphere; two of the boundary components are isotopic to $a$ and $b$. The {\em connected sum} $a+_\epsilon b$ is the simple closed curve in the isotopy class of the third boundary component. See Figure \ref{figure:casum}.
\end{definition}

\begin{figure}
\labellist
\Huge
\pinlabel $\rightsquigarrow$ at 180 45
\endlabellist
\includegraphics{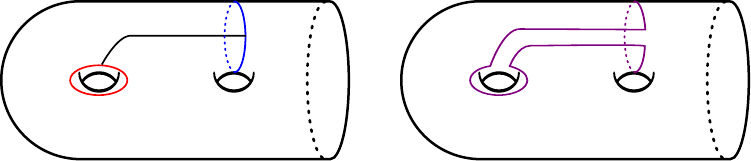}
\caption{The connected sum operation.}
\label{figure:casum}
\end{figure}

\begin{lemma}\label{lemma:ctwist}
Let $(x_i, y_i)$ and $(x_j,y_j)$ be distinct pairs of symplectic basis vectors, and let $z \in H$ be a primitive vector orthogonal to $\langle x_i,y_i,x_j,y_j\rangle$; if $r$ is even, suppose $q(z) = 1$. Then there is an element $f \in \T_\phi \cap \I_g$ satisfying
\[
\tau(f) = z \wedge (x_i \wedge y_i - x_j \wedge y_j).
\]
\end{lemma}
\begin{proof}
\begin{figure}
\labellist
\small
\pinlabel $a$ [bl] at 24 102
\pinlabel $b$ [bl] at 211.2 102
\pinlabel $c$ [bl] at 120 102
\pinlabel $x_i$ [bl] at 72 73.4
\pinlabel $y_i$ [l] at 65 26.8
\pinlabel $x_j$ [bl] at 170 72.4
\pinlabel $y_j$ [l] at 161.6 26.8
\endlabellist
\includegraphics{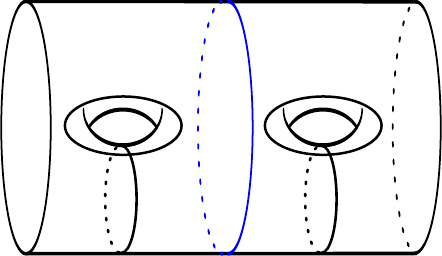}
\caption{The configuration of curves used to exhibit (G\ref{item:difference}).}
\label{figure:g2}
\end{figure}
This follows the argument for (G\ref{item:difference}) given in \cite[proof of Lemma 5.8]{Salter_monodromy}. The change--of--coordinates principle (in the guise of Lemma \ref{lemma:ccphomol}) implies that there exists an admissible curve $c$ such that $[c] =z$ (in the case of $r$ even, this uses the assumption that $q(z) = 1$). By the classical change--of--coordinates principle, there exist curves $a,b$ with the following properties: (1) $a \cup b$ bounds a subsurface $S$ of genus $2$, (2) $c \subset S$, (3) $[a] = [b] = [c]$, and $c$ separates $S$ into two subsurfaces $S_1, S_2$ each of genus $1$, (4) $x_i, y_i$ determine a symplectic basis for $S_1$ and $x_j, y_j$ determine a symplectic basis for $S_2$. Such a configuration is shown in Figure \ref{figure:g2}. By homological coherence (Lemma \ref{lemma:homcoh}), $\phi(a) = \phi(b) = -2$ when $a,b$ are oriented with $S$ to the left. 
By Lemma \ref{lemma:johnsonhom}.\ref{item:bpformula},
\[
\tau(T_a T_c^{-1}) = z \wedge x_i \wedge y_i
\]
and
\[
\tau(T_b T_c^{-1}) = -z \wedge x_j \wedge y_j.
\]
Therefore, it is necessary to show $T_a T_b T_c^{-2} \in \mathcal T_\phi$. By hypothesis, $T_c \in \mathcal T_\phi$, so it remains to show $T_a T_b \in \T_\phi$ as well. 

We claim that there exists a maximal chain $a_1, \dots, a_5$ of admissible curves on $S$; modulo this, the claim follows by an application of the chain relation. Choose an arbitrary subsurface $S' \subset S$ homeomorphic to $\Sigma_2^1$, and let $\mathcal B = \{c_1,c_2,c_3,c_4\}$ be a geometric symplectic basis for $S'$; by Lemma \ref{lemma:ccpodd}.\ref{item:oddgsb} or \ref{lemma:ccpeven}.\ref{item:gsbeven}, such a basis can be chosen with $c_1,c_2,c_3$ admissible, and $c_4$ either admissible or else satisfying $\phi(c_4) = -1$. 

If $\phi(c_4) = 0$, consider the connected sum
\[
c_4' = c_4 +_{\epsilon} a,
\]
where $\epsilon$ is disjoint from $c_1,c_2, c_3$. By homological coherence (Lemma \ref{lemma:homcoh}), $\phi(c_4') = -1$, and $\{c_1, c_2, c_3, c_4'\}$ forms a geometric symplectic basis. Thus we may assume that $S'$ is chosen with $c_1,c_2,c_3$ admissible and $\phi(c_4) = -1$. Under this assumption, we can set $a_1 = c_1, a_2 = c_2, a_4 = c_3$, and $a_3 = c_4 +_\epsilon a_2$ with $\epsilon$ an arc connecting the left side of $c_4$ to $a_1$ and otherwise disjoint from the other curves under consideration. By homological coherence, $a_3$ is admissible as well. Now let $a_5$ be any curve extending $a_1,\dots, a_4$ to a maximal chain on $S$. Since $\phi(a) = -2$ when oriented with $S$ to the left, homological coherence implies that $a_5$ is admissible, and we have constructed the required maximal chain. 
\end{proof}

\para{Concluding Lemma \ref{lemma:refine2}} We can now show that $\T_\phi \cap \I_g$ surjects onto the kernel of the contraction $C_s$. Note that establishing Lemma \ref{lemma:Texhibit} will complete the proof of Lemma \ref{lemma:refine2}, which in turn completes the final Step 3 of the proof of Proposition \ref{prop:admissfull}.

\begin{lemma}\label{lemma:spgenset}
For any $s$ dividing $g-1$, the subspace $\ker(C_s) \le \wedge^3 H/H$ has a generating set consisting of the following classes of elements; in each case $z \in \{x_1, y_1, \dots, x_g, y_g\}$ with further specifications listed below.
\begin{enumerate}[(G1)]
\item\label{item:deltatwist} $s (z \wedge x_i \wedge y_i)$ for $z \ne x_i, y_i$
\item\label{item:difference} $z \wedge (x_i \wedge y_i - x_j\wedge y_j)$ for $z \ne x_i, y_i, x_j, y_j$
\item\label{item:lagrangian} $z_i \wedge z_j \wedge z_k$ for $\{i,j,k\} \subset \{1, \dots, g\}$ distinct. 
\end{enumerate}
\end{lemma}
\begin{proof}
See \cite[proof of Lemma 5.8]{Salter_monodromy}.
\end{proof}

\begin{lemma}\label{lemma:Texhibit}
Let $\phi$ be an $r$--spin structure on a surface $\Sigma_g$ of genus $g \ge 3$ (if $g = 3$, assume $\Arf(\phi) = 1$). For each generator $f$ of the form (G\ref{item:deltatwist}) -- (G\ref{item:lagrangian}) as presented in Lemma \ref{lemma:spgenset}, the group $\T_\phi \cap \I_g$ contains an element $\gamma$ satisfying $\tau(\gamma) =f$. 
\end{lemma}

\begin{proof}
To avoid having to formulate two nearly identical arguments, one for each parity of $r$, we treat only the case of $r$ even. The presence of a residual mod--$2$ spin structure makes this case strictly harder than that for $r$ odd. 

Let $q$ denote the quadratic form associated to $\phi$; recall that if $c$ is a simple closed curve, then $q([c]) = \phi([c]) + 1 \pmod 2$. We fix a symplectic basis $\mathcal B = \{x_1, y_1, \dots, x_g, y_g\}$ such that $q(x_i) = 1$ for $1 \le i \le g$ and $q(y_i)= 1$ for $1 \le i \le g-1$; the value of $q(y_g)$ is then determined by $\Arf(q)$. Throughout, we will use the following principle: we will perform a topological computation to obtain some tensor in $\tau(\T_\phi \cap \I_g) \le \wedge^3H/H$. We will then combine the $\Sp(2g,\Z)$--equivariance of Lemma \ref{lemma:johnsonhom}.1 and the surjectivity of the homological representation $\Psi$ onto $\Sp(2g,\Z)[q]$ (Lemma \ref{step1}) to see that this single computation provides a large class of further elements of $\tau(\T_\phi \cap \I_g)$.

Generators of type (G\ref{item:deltatwist}) are of the form $s(z \wedge x_i \wedge x_j)$; here $s = r/2$. To obtain such elements in $\tau(\T_\phi \cap \I_g)$, we begin by using the change--of--coordinates principle (Lemma \ref{lemma:ccpeven}.4) to choose a $3$--chain of admissible curves $a_0, a_1, a_2$ representing respectively the homology classes $x_1, y_1, (x_1 + x_2 + x_3)$. Let $b,b'$ denote the boundary components of this chain. By the chain relation, $T_b T_{b'} \in \T_\phi$ and so 
\[
(T_b T_{b'})^s \in \T_\phi
\]
as well. By homological coherence, $\phi(b) = -1$, and so by Lemma \ref{lemma:twistpowers}, also 
\[
T_b^{r} \in \T_\phi
\]
Combining these two shows that
\[
(T_{b} T_{b'}^{-1})^{s} \in \T_\phi.
\]
By Lemma \ref{lemma:johnsonhom}.\ref{item:bpformula}, it follows that
\[
s\ (x_1 \wedge y_1\wedge (x_2 + x_3)) \in \tau(\T_\phi \cap \I_g).
\]
This argument can be repeated with curves $a_0', a_1', a_2'$ representing respectively the homology classes $x_1, y_1, (x_1 + y_2 + x_3)$, showing that also
\[
s\ (x_1 \wedge y_1\wedge (y_2 + x_3)) \in \tau(\T_\phi \cap \I_g).
\]
Subtracting,
\[
s\ (x_1 \wedge y_1 \wedge (x_2 - y_2)) \in \tau(\T_\phi \cap \I_g).
\]
As $q(x_1) = q(y_1) = q(x_2-y_2) = 1$, Lemma \ref{lemma:symptools}.\ref{item:contractible} shows that $\tau(\T_\phi \cap \I_g)$ contains all generators of type \eqref{item:deltatwist} of the form $s\ (z \wedge x_i \wedge y_i)$ for $i \le g-1$, except for $z = y_g$ in the case $q(y_g) = 0$. In this latter case, an application of Lemma \ref{lemma:symptools}.\ref{item:contractible} to $s\ (x_1 \wedge y_1 \wedge (x_2 + x_3))$ shows that $s\ (y_g \wedge x_i \wedge y_i) \in \tau(\T_\phi \cap \I_g)$ for $i \le g-1$ regardless. 

It remains to show $s\ (z \wedge x_g \wedge y_g) \in \tau(\T_\phi \cap \I_g)$. If $q(y_g)= 1$ then the above results are already sufficient. Otherwise, by above, 
\[
s\ (x_1 \wedge x_g \wedge (y_{g-1} + y_g)) \in \tau(\T_\phi \cap \I_g).
\]
It thus suffices to show $s\ (x_1 \wedge x_g \wedge y_{g-1}) \in \tau(\T_\phi \cap \I_g)$. By Lemma \ref{lemma:symptools}.\ref{item:threetuple}, it is in turn sufficient to show $s\ (x_1 \wedge x_2 \wedge x_3) \in \tau(\T_\phi \cap \I_g)$. By the computations above,
\[
s\ (x_1 \wedge (y_1 + x_2 + x_3) \wedge x_3)  \qquad \mbox{ and }\qquad s\ (x_1 \wedge y_1 \wedge x_3)
\]
are both elements of $ \tau(\T_\phi \cap \I_g)$; taking the difference, the result follows. 

Now we consider generators of type (G\ref{item:difference}); recall these are of the form $z \wedge (x_i \wedge y_i - x_j \wedge y_j)$. Applying Lemma \ref{lemma:ctwist}, we find 
\[
z \wedge (x_i \wedge y_i - x_j \wedge y_j) \in \tau(\T_\phi \cap \I_g)
\]
for $x_i,y_i,x_j,y_j$ arbitrary and for all $z \in \{x_1, \dots, y_g\}$ satisfying $q(z) = 1$. This encompasses all elements $x_1, \dots, x_{g-1}, y_{g-1}, x_g$, and possibly $y_g$ as well. In the case where $q(y_g) = 0$, we have $q(y_{g-1}) = q(y_{g-1}+ y_g) = 1$. Applying Lemma \ref{lemma:ctwist} with $z = y_{g-1}$ and $z = y_{g-1}+y_g$ in turn and subtracting, we obtain all elements of the form 
\[
y_g \wedge (x_i \wedge y_i - x_j \wedge y_j) \in \tau(\T_\phi \cap \I_g)
\]
as well, completing this portion of the argument. 

Finally, we consider generators of type (G\ref{item:lagrangian}), of the form $z_i\wedge z_j \wedge z_k$ for distinct indices $i,j,k$. By Lemma \ref{lemma:ccphomol}, there exists a curve $d$ with $[d] = x_2$ and $\phi(d)= -2$. Choose some curve $e_1$ disjoint from $d$ such that $d \cup e_1$ bounds a subsurface $S_1$ of genus $1$ to the left of $d$, and such that $S_1$ contains a pair of curves in the homology classes $x_1, y_1$. By homological coherence (Lemma \ref{lemma:homcoh}), $e_1$ is admissible, and by Lemma \ref{lemma:johnsonhom}.\ref{item:bpformula},
\[
\tau(T_d T_{e_1}^{-1}) = x_1 \wedge y_1 \wedge x_2.
\]
Similarly, we can find a curve $e_2$ disjoint from $d$ such that $d \cup e_2$ bounds a subsurface $S_2$ of genus $1$ to the left of $d$, and such that $S_2$ contains a pair of curves in the homology classes $x_1, y_1-x_3$. Again by homological coherence, $e_2$ is admissible, and by Lemma \ref{lemma:johnsonhom}.\ref{item:bpformula},
\[
\tau(T_d T_{e_2}^{-1}) = x_1 \wedge (y_1-x_3) \wedge x_2.
\]
Combining these computations, since $e_1, e_2$ are admissible,
\[
\tau(T_{e_2} T_{e_1}^{-1}) = x_1 \wedge x_3 \wedge x_2 \in \tau(\T_\phi \cap \I_g).
\]

Applying Lemma \ref{lemma:symptools}.\ref{item:threetuple}, it follows that if $z_i,z_j,z_k$ are pairwise--orthogonal primitive vectors with $q(z_i) = q(z_j) = q(z_k) = 1$, then $z_i \wedge z_j \wedge z_k \in \tau(\T_\phi \cap \I_g)$. This includes all generators of the form (G\ref{item:lagrangian}) except when $z_i = y_g$ and $q(y_g) = 0$. In this case, both $q(y_{g-1}) = q(y_{g-1} + y_g) = 1$, and we conclude the argument as we did for (G\ref{item:difference}) by finding $z_i \wedge z_j \wedge y_{g-1}$ and $z_i \wedge z_j \wedge (y_{g-1} + y_g)$ in $\tau(\T_\phi \cap \I_g)$ and subtracting.
\end{proof}

\para{Recap: Completing the proof of Proposition \ref{prop:admissfull}} Having at this point completed Step 3, we have now established all of the claims necessary to prove Proposition \ref{prop:admissfull}, as outlined above in Section \ref{subsection:firststep}. We give a final summary of the proof below.
\begin{proof}[Proof of Proposition \ref{prop:admissfull}]
Recall that the objective is to show that for an arbitrary $r$-spin structure $\phi$ on a surface of genus $g \ge 3$, there is an equality $\mathcal T_\phi = \Mod_g[\phi]$ (recall that $\mathcal T_\phi$ is the admissible subgroup, i.e. the group generated by twists about admissible curves). By definition, there is a containment
\[
\mathcal T_\phi \le \Mod_g[\phi].
\]
By Step 2 (Lemma \ref{step2}), the Johnson kernel $\mathcal K_g$ is a subgroup of $\mathcal T_\phi$, and so it will suffice to show that there is an equality
\[
\mathcal T_\phi / \mathcal K_g = \Mod_g[\phi] / \mathcal K_g. 
\]
The group $\Mod_g[\phi] / \mathcal K_g$ fits into the short exact sequence below (recall that $\tau$ is the Johnson homomorphism and $\Psi$ is the symplectic representation, both discussed above in Section \ref{subsection:firststep}):
\[
1 \to \tau(\Mod_g[\phi] \cap \mathcal I_g) \to \Mod_g[\phi] / \mathcal K_g \to \Psi(\Mod_g[\phi]) \to 1.
\]
Step 3 (Lemma \ref{step3}) establishes the equality $\tau(\Mod_g[\phi] \cap \mathcal I_g) = \tau(\mathcal T_\phi \cap \mathcal I_g)$; along with the containments $\ker(\tau) = \mathcal K_g \le \mathcal T_\phi \le \Mod_g[\phi]$, this shows that 
\[
\Mod_g[\phi] \cap \mathcal I_g =  \mathcal T_\phi \cap \mathcal I_g.
\]
Thus it remains only to show that $\Psi(\Mod_g[\phi]) = \Psi(\mathcal T_\phi)$ as subgroups of $\Sp(2g,\Z)$, and this is established in Step 1 (Lemma \ref{step1}).
\end{proof}

In turn, the completion of Proposition \ref{prop:admissfull} allows us to complete the proof of Theorem \ref{theorem:genset} in the setting $r = 2g-2$.

\begin{proof}[Proof of Theorem \ref{theorem:genset}\ref{case:hardeven} and \ref{theorem:genset}\ref{case:hardodd}]
By Proposition \ref{lemma:gammaadmiss}, the set of Dehn twists in the curves indicated in Figure \ref{figure:hardcaseeven} (respectively, Figure \ref{figure:hardcaseodd}) generates $\mathcal T_\phi$, where $\phi$ is the $(2g-2)$--spin structure specified by assigning $\phi(c) = 0$ for every curve $c$. By Proposition \ref{prop:admissfull}, $\mathcal T_\phi = \Mod_g[\phi]$. Therefore the Dehn twists in the curves of Figure \ref{figure:hardcaseeven} (respectively, Figure \ref{figure:hardcaseodd}) generate the respective spin structure stabilizers.
\end{proof}

\subsection{The case of general $r$}\label{section:generalr}

In this section, we prove Theorem \ref{theorem:genset}\ref{case:r}.
We first demonstrate how the change--of--coordinates principle and twist--linearity can be used, given two curves, to produce a third whose winding number is the greatest common divisor of the other two.

\begin{lemma}\label{lemma:gcd}
Let $\phi$ be a $({2g-2})$--spin structure on a surface $\Sigma_g$ of genus at least $3$ and suppose that $ \phi(a_1) = k_1$ and $ \phi(a_2) = k_2$. Set
\[\Gamma = \langle \Mod_g[\phi], T_{a_1}, T_{b_2} \rangle.\]
Then $\Gamma$ contains $T_c$ for some nonseparating curve $c$ with $ \phi(c) = \gcd(k_1, k_2).$
\end{lemma}
\begin{proof}
Set $r = \gcd(k_1, k_2)$; then there exist some $x, y \in \Z$ such that
\[xk_1 + yk_2 = r.\]
Without loss of generality, we may suppose that $x,y \neq 0$ (else $T_{a_1}$ or $T_{a_2}$ has the desired property.).

By the change--of--coordinates principle (Lemma \ref{lemma:ccpeven}), there exists some curve $b_1$ with $i(b_1, a_2) = 1$ and $ \phi(b_1) = k_1$. By Lemma \ref{lemma:ccpcurves}, there is some element $f \in \Mod_g[\phi]$ such that $f(a_1)=b_1$; then $T_{b_1} = f T_{a_1} f^{-1}$ and so $T_{b_1} \in \Gamma$.
Now by twist linearity (Definition \ref{definition:spin}(\ref{item:TL})), we have that
\[ \phi \left( T_{b_1}^x (a_2) \right)
=  \phi(a_2) + x \phi( b_1) = xk_1 + k_2.
\]
Again by Lemma \ref{lemma:ccpeven} , there is a curve $b_2$ which only intersects $T_{b_1}^x (a_2)$ once and has $ \phi(b_2) = k_2$. Applying Lemma \ref{lemma:ccpcurves} as above, we similarly see that $T_{b_2} \in \Gamma$. Therefore
\[
 \phi \left( T_{b_2}^{( y-1)}\left( T_{b_1}^x (a_2) \right) \right)
=  \phi \left( T_{b_1}^x (a_2) \right) + (y-1) \phi( b_2)
= xk_1 + yk_2 = r.
\]
Setting $c=T_{b_2}^{( y-1)} T_{b_1}^x (a_2)$ completes the proof (since $c$ is in the $\Gamma$ orbit of $a_2$, the twist $T_c$ is conjugate to $T_{a_2}$ by an element of $\Gamma$).
\end{proof}

\begin{proof}[Proof of Theorem \ref{theorem:genset}\ref{case:r}]
Let $\phi$, $\tilde \phi$, and $\{c_i\}$ be as in the statement of the theorem, and set 
\[
\Gamma = \pair{\Mod_g[\tilde \phi], \{T_{c_i}\}}.
\]
Our task is to show that $\Gamma = \Mod_g[\phi]$. By Proposition \ref{prop:admissfull}, it is enough to show that $\Gamma = \T_\phi$, i.e. that $T_c \in \Gamma$ for any curve $c$ with $\phi(c)=0$. Since $\tilde \phi$ is a lift of $\phi$, we see that this is equivalent to showing that $T_c \in \Gamma$ for every $c$ with $\tilde \phi(c) \equiv 0 \pmod r$.

We claim that it is enough to exhibit the Dehn twist in a single curve $c$ with $\tilde \phi(c)=r$. Indeed, by the transitivity of $\Mod_g[\tilde \phi]$ on the set of (non--separating) curves with given $\tilde \phi$--winding number (Lemma \ref{lemma:ccpcurves}), this implies that every curve with winding number $r$ is in $\Gamma$. Now given any $c$ with $\tilde \phi(c) \equiv 0 \pmod r$, Lemma \ref{lemma:ccpeven}.4 guarantees that there exists some curve $d$ with $\tilde \phi(d) = r$ which intersects $c$ exactly once, and by twist--linearity (Definition \ref{definition:spin}), we have that
\[ \tilde \phi \left( T_{d}^{-(\tilde \phi(c) / r)} (c) \right) 
= \tilde \phi(c) - (\tilde \phi(c) / r) \cdot r = 0.\]
Therefore $T_c$ may be conjugated to a $\tilde \phi$--admissible twist by an element of $\Gamma$ and hence $T_c \in \Gamma$.

To exhibit such a twist, one needs only to iteratively apply Lemma \ref{lemma:gcd} to the collection of curves $\{ c_i \}$ to recover a curve with $\tilde \phi$--winding number $r = \gcd \left( \tilde \phi(c_i) \right)$.
\end{proof}

\section{Connected components and geometric monodromy groups of strata}\label{section:classification}

In this section, we prove Theorems \ref{theorem:classification} and \ref{theorem:moncomp}. Our strategy matches the one employed by the first author in \cite[\S\S 5,6]{Calderon_strata}, but we reproduce the details below for the convenience of the reader.

The plan of proof is as follows: appealing to Kontsevich and Zorich's classification of the components of $\HM(\sing)$ (Theorem \ref{theorem:KZstrata}), we equate the classification of components of $\HT(\sing)$ with the computation of the geometric monodromy groups (Definition \ref{def:geomonodromy}) of components of $\HM(\sing)$. This is recorded as Proposition \ref{prop:monstab}.

Ultimately, Theorem \ref{theorem:moncomp} implies that each geometric monodromy group {\em coincides} with the stabilizer of some $r$-spin structure, and so by the orbit--stabilizer theorem, $r$--spin structures are in 1--to--1 correspondence with the (non--hyperelliptic) components of $\HT(\sing)$. Lemma \ref{lemma:rspinconstant} shows that each geometric monodromy group stabilizes some $r$--spin structure $\phi$, demonstrating one direction of inclusion.

The reverse inclusion follows by applying Theorem \ref{theorem:genset}. In Section \ref{subsection:prototype}, we use a construction of Thurston and Veech to build an abelian differential out of a system of curves satisfying the hypotheses of Theorem \ref{theorem:genset} (Lemma \ref{lemma:prototype}). The geometry of these differentials allows us to realize the Dehn twist in each curve as a continuous deformation of the flat structure, thereby implying that every element of $\Mod_g[\phi]$ is realized through flat deformations and so proving Theorem \ref{theorem:moncomp}, hence Theorem \ref{theorem:classification}.

\subsection{The flat geometry of an abelian differential}\label{subsection:flat}

We begin with a review of some background information on abelian differentials and their induced flat cone metrics. While not every flat cone metric comes from an abelian differential (Lemma \ref{lemma:flatequiv}), those that do induce $r$--spin structures (Construction \ref{example:diffspin}). In Lemma \ref{lemma:cylinderadmissible}, we record a first result relating the geometry of the flat metric with the admissible curves for the induced spin structure.

An {\em abelian differential} $\omega$ on a Riemann surface $X$ is a holomorphic section of the canonical bundle $K_X$. In charts away from its zeros $\{p_1, \ldots, p_n\}$, the form is locally equivalent to $dz$, while at each $p_i$ it is locally equivalent to $z^{k_i} dz$ for some $k_i \ge 1$. The metric given by $|dz|^2$ is then Euclidean away from each $p_i$, at which the metric has a cone angle of $2\pi (k_i+1)$.
Along with the flat cone metric, $\omega$ also induces a ``horizontal'' vector field $V=1/\omega$ away from $\{p_1, \ldots, p_n\}$, at which $V$ has index $-k_i$. For a more thorough discussion, see, e.g., \cite{Zorich_Survey}.

In practice, it is often useful not to build abelian differentials directly, but instead to build flat metrics and then check that they are induced from abelian differentials. For any locally flat metric $\sigma$ on a closed surface $\Sigma_g$ with finitely many cone points $p_1, \ldots, p_n$, there is a natural {\em holonomy representation}
\[\text{hol}_\sigma: \pi_1(\Sigma \setminus \{p_1, \ldots, p_n\} ) \rightarrow SO(2)\]
which measures the rotational difference between a tangent vector and its parallel transport along a loop in $\Sigma \setminus \{p_1, \ldots, p_n\}$.

It is a standard fact that the holonomy of a locally flat metric determines whether or not it comes from an abelian differential (see, e.g., \cite[\S 1.8]{MasurTabach_Survey}).

\begin{lemma}\label{lemma:flatequiv}
Let $\sigma$ be a flat cone metric on a closed surface $\Sigma_g$ with cone points $\{p_1, \ldots, p_n\}$. Then the following are equivalent:
\begin{enumerate}
\item There exists some Riemann surface $X$ and abelian differential $\omega$ so that $(X, \omega)$ with the induced metric $|dz|^2$ is isometric to $(\Sigma, \sigma)$
\item The holonomy representation $\text{hol}_\sigma:\pi_1(\Sigma \setminus \{p_1, \ldots, p_n\} ) \rightarrow SO(2)$ is trivial.
\item The cone angle at each $p_i$ is $2\pi(k_i+1)$ for some $k_i \in \mathbb{N}$ and there exists a locally constant vector field $V$ on $(\Sigma, \sigma)$, singular only at the $p_i$ with index $-k_i$ at each $p_i$.
\end{enumerate}
\end{lemma}

A vector field $V$ as above is sometimes called a {\em translation vector field}. The winding number with respect to a translation vector field serves as our bridge between the flat geometry of an abelian differential and its induced $r$--spin structure.

\begin{construction}[c.f. Example \ref{example:wnf}]\label{example:diffspin}
Suppose that $X$ is a Riemann surface equipped with an abelian differential $\omega$ with zeros of order $(k_1, \ldots, k_n)$. Let $V = 1/\omega$. Then the winding number function of a curve with respect to $V$ defines an $r$--spin structure $\phi$ on $X$, where $r = \gcd(k_1, \ldots, k_n)$ (see Example \ref{example:wnf}).

If $X$ is also equipped with a {\em marking}, that is, an isotopy class of map $f: \Sigma \rightarrow X$ where $\Sigma_g$ is a reference topological surface, then $\phi$ pulls back to an $r$--spin structure on $\Sigma_g$, which by abuse of notation we will also denote by $\phi$. In this way, we see that any marked abelian differential $(X, f, \omega)$ gives rise to an $r$--spin structure on $\Sigma_g$.
\end{construction}

\begin{remark}
It is not hard to see that any two translation vector fields on $(X, \omega)$ are related by a rotation, and so any two translation vector fields will induce the same $r$--spin structure as $V = 1/ \omega$.
\end{remark}

Because of the relationship between the flat geometry of $(X, \omega)$ and the vector field $V$, it is easy to produce examples of curves on $X$ of winding number $0$. If $\Sigma_g$ is a surface equipped with a flat cone metric $\sigma$, then a {\em cylinder} on $(\Sigma, \sigma)$ is an embedded Euclidean cylinder which contains no cone points in its interior.

\begin{lemma}[Lemma 4.6 of \cite{Calderon_strata}]\label{lemma:cylinderadmissible}
If $(X, \omega)$ is an abelian differential defining an $r$--spin structure $\phi$, then the core curve of any cylinder on $(X, \omega)$ is $\phi$--admissible.
\end{lemma}

\subsection{Geometric monodromy and cylinder twists} Lemma \ref{lemma:cylinderadmissible} provides a first point of contact between the flat geometry of an abelian differential and the algebra of a spin structure stabilizer. In this subsection, we deepen this link, recording in Theorem \ref{theorem:KZstrata} Kontsevich and Zorich's classification of components of strata over moduli space. We then show in Lemma \ref{lemma:rspinconstant} that {\em any} deformation of a marked abelian differential preserves the associated $r$--spin structure. In Lemma \ref{lemma:cyltwist}, we begin the process of establishing the converse (and hence Theorem \ref{theorem:moncomp}) with the elementary but crucial observation that each cylinder gives rise to an admissible twist.

The relationship between abelian differentials and spin structures is not new; indeed, it plays a pivotal role in Kontsevich and Zorich's classification of the components of strata of unmarked differentials. In the absence of a marking, it is impossible to compare the spin structures induced by two points $(X, \omega)$ and $(X', \omega')$. It is therefore the Arf invariant, not the spin structure, which serves as a classifying invariant for the (non-hyperelliptic) components of $\HM(\sing)$.

\begin{theorem}[Theorem 1 of \cite{KZ_strata}]\label{theorem:KZstrata}
Let $g \ge 4$ and $\sing = (k_1, \ldots, k_n)$ be a partition of $2g-2$. Then there are always at most $3$ components of $\HM(\sing)$, which are classified as follows:
\begin{itemize}
\item If $\gcd(\sing)$ is odd, then $\HM(\sing)$ has exactly one non-hyperelliptic connected component.
\item If $\gcd(\sing)$ is even, then $\HM(\sing)$ has exactly two non-hyperelliptic components, distinguished by the Arf invariant of the 2--spin structure induced by $(X, \omega)$.
\item If $\gcd(\sing) \in \{2g-2, g-1\}$, then in addition to the non-hyperelliptic components above, there is one further connected component of $\HM(\sing)$ that consists entirely of hyperelliptic differentials.
\end{itemize}
\end{theorem}

With this classification in hand, we can now record how deformations of an abelian differential give rise to mapping classes. 
Note that each component of $\HM(\sing)$ is generally not a manifold but an orbifold, which is locally modeled on $H^1(X, \text{Zeros}(\omega); \C)$ away from its singularities (see, e.g., \cite[Lemma 2.1]{Wright_Survey1}).

\begin{definition}\label{def:geomonodromy}
Let ${ \mathcal H}$ be a component of $\HM(\sing)$ and let $(X, \omega) \in { \mathcal H}$ be a generic (non-orbifold) point. The forgetful map from $\mathcal{H} \rightarrow \mathcal{M}_g$ induces a map on fundamental groups
\[\rho:\pi_1^{\text{orb}}({ \mathcal H}, (X, \omega)) \rightarrow \Mod(X).\]
whose image is called the {\em geometric monodromy group} $\mathcal{G}({ \mathcal H})$ of $\mathcal H$.
\end{definition}

\begin{remark}\label{rem:geomonodromy}
One may also heuristically think of this map as coming from the parallel transport of the underlying topological surface along a loop $\gamma$ in ${ \mathcal H}$, yielding a(n isotopy class of) self--diffeomorphism $\rho(\gamma)$ of $X$. This description is unfortunately not quite correct, as there is no universal family over a stratum (due to to the presence of translation automorphisms).

This intuition can be made precise by noting that the orbifold points of $\mathcal H$ map to orbifold points of $\mathcal M_g$ with compatible isotropy groups, since any symmetry of $(X, \omega)$ is in particular a symmetry of $X$. The space $\HT(\sing)$ is therefore a manifold, equipped with a $\Mod_g$--equivariant forgetful map to $\T_g$. Picking any lift $\tilde{\gamma}$ of $\gamma$ to $\HT(\sing)$, the element $\rho(\gamma)$ can then be defined as the change-of-marking map from $\tilde{\gamma}(0)$ to $\tilde{\gamma}(1)$.
\end{remark}

Fixing a marking $f: \Sigma_g \rightarrow X$ identifies $\mathcal{G}({ \mathcal H})$ as a subgroup of $\Mod(\Sigma_g)$, but choosing a different basepoint $(X, \omega)$ or different marking $f$ will conjugate $\mathcal{G}({ \mathcal H})$ in $\Mod(\Sigma_g)$.
As such, in the following statements we will always begin by fixing some marked abelian differential $(X, f, \omega)$ with $(X, \omega) \in { \mathcal H}$; then the geometric monodromy group should be understood to be defined with reference to basepoint $(X, \omega)$ and marking $f$.

\begin{proposition}\label{prop:monstab}
Let ${ \mathcal H}$ be a component of $\HM(\sing)$ and fix $(X, f, \omega)$ with $(X, \omega) \in { \mathcal H}$. Then the components of $\HT(\sing)$ that cover ${ \mathcal H}$ are in bijective correspondence with the cosets of $\mathcal{G}({ \mathcal H})$ in $\Mod(\Sigma_g)$.
\end{proposition}
\begin{proof}
Let $\tilde { \mathcal H}$ be the component of $\HT(\sing)$ containing $(X, f, \omega)$. The mapping class group acts on the set of components of $\HT(\sing)$ which cover ${ \mathcal H}$ by permutations, and so it suffices to show that 
\[\mathcal{G}({ \mathcal H}) = \text{Stab}_{\Mod(\Sigma_g)}(\tilde { \mathcal H}).\]
Now if a mapping class $g$ is in $\mathcal{G}({ \mathcal H})$ then it is the image of a loop $\gamma$ in $\HM(\sing)$, which can be lifted to a path $\tilde{\gamma}$ in $\HT(\sing)$ from $(X,f, \omega)$ to $g \cdot (X, f, \omega)$. Therefore $g \in \text{Stab}_{\Mod(\Sigma_g)}(\tilde { \mathcal H})$.

Conversely, if $g$ stabilizes $\tilde { \mathcal H}$, then since $\tilde { \mathcal H}$ is path--connected there is a path $\tilde{\gamma}$ in $\tilde { \mathcal H}$ from $(X, f, \omega)$ to $g \cdot (X, f, \omega)$. The projection of $\tilde{\gamma}$ to ${ \mathcal H}$ is a loop $\gamma$ whose geometric monodromy is exactly $g$, and hence $g \in \mathcal{G}({ \mathcal H})$.
\end{proof}

Because the horizontal vector field of $(X, \omega)$ deforms continuously along with $(X, \omega)$, the winding number of any curve on $X$ is constant and so the geometric monodromy group must preserve the induced $r$--spin structure.

\begin{lemma}[Corollary 4.8 in \cite{Calderon_strata}]\label{lemma:rspinconstant}
Let $g \ge 2$ and $\sing$ a partition of $2g-2$ with $\gcd(\sing) = r$. Let ${ \mathcal H}$ be a component of $\HM(\sing)$ and fix $(X, f, \omega)$ with $(X, \omega) \in { \mathcal H}$. Then
\[\mathcal{G}({ \mathcal H}) \le \Mod_g[\phi],\]
where $\phi$ is the $r$--spin structure corresponding to $(X, f, \omega)$.
\end{lemma}

To exhibit the reverse inclusion, we need a way to build elements of $\mathcal{G}({ \mathcal H})$. A particularly simple method is to realize curves as cylinders; then the corresponding Dehn twists can be realized as continuous deformations of flat surfaces, and hence as elements of $\mathcal{G}({ \mathcal H})$.

\begin{lemma}[c.f. Lemma 6.2 in \cite{Calderon_strata}]\label{lemma:cyltwist}
Let ${ \mathcal H}$ be a component of $\HM(\sing)$ and fix $(X, f, \omega)$ with $(X, \omega) \in { \mathcal H}$. If $c$ is a simple closed curve on $\Sigma_g$ such that $f(c)$ is the core curve of a cylinder on $(X, \omega)$, then $T_c \in \mathcal{G}({ \mathcal H}).$
\end{lemma}

\begin{remark}
Lemma 6.2 of \cite{Calderon_strata} only deals with the case when ${ \mathcal H}$ is a non-hyperelliptic component of $\HM(\sing)$. When ${ \mathcal H}$ consists entirely of hyperelliptic differentials, the result follows from the description of $\mathcal{G}({ \mathcal H})$ appearing in the proof of \cite[Corollary 2.6]{Calderon_strata} together with the fact that the hyperelliptic involution of any $(X, \omega) \in { \mathcal H}$ (setwise) fixes each of its cylinders (see, e.g., \cite[Lemma 2.1]{Lind_hyp}).
\end{remark}

\subsection{Construction of prototypes and the proof of the main theorems}\label{subsection:prototype}

We now recall the Thurston--Veech method for building a flat surface out of a filling pair of multicurves. In Lemma \ref{lemma:flatprototype}, we use this procedure to build a locally flat metric with a collection of cylinders whose core curves satisfy the hypotheses of Theorem \ref{theorem:genset}. In Lemma \ref{lemma:prototype}, we analyze the holonomy of this metric in order to show that this is induced from an abelian differential. This is used to deduce that the geometric monodromy group of a component of $\HM(\sing)$ is exactly the stabilizer of the corresponding $r$--spin structure, completing the proof of Theorems \ref{theorem:classification} and \ref{theorem:moncomp}.

In Definition \ref{def:curvesystem} below, we describe a system of curves such that Dehn twists about these curves generate an $r$-spin mapping class group $\Mod_g[\phi]$. As observed in Lemma \ref{lemma:stabindex}, if $r$ is odd, then $\Mod(\Sigma_g)$ acts transitively on the set of $r$-spin structures, and if $r$ is even, then there are two orbits, distinguished by the Arf invariant. Thus there are only ever at most two conjugacy classes of $r$-spin mapping class groups for a given genus. The three cases in Definition \ref{def:curvesystem} therefore exhaust all possible behaviors with regards to the existence and parity of Arf. The dichotomy between $g \equiv 0,1 \pmod 4$ and $g \equiv 2,3 \pmod 4$ is a reflection of the fact that the Arf invariant of both of the configurations shown in Figures \ref{figure:curvelabels12} and \ref{figure:curvelabels3} exhibit mod-$4$ periodicity as a function of $g$, and are constructed so that the Arf invariants for a fixed $g$ are always distinct.

\begin{definition}[Definition 5.1 of \cite{Calderon_strata}]\label{def:curvesystem}
Suppose that $g \ge 3$ and let $\sing = (k_1, \ldots, k_n)$ be a partition of $2g-2$. If $\gcd(\sing)$ is even, also choose $\Arf \in \{0,1\}$. Label curves of $\Sigma_g$ as follows:
\begin{enumerate}
\item If $\gcd(\sing)$ is odd, then label the curves as in Figure \ref{figure:curvelabels12}

\item If $\gcd(\sing)$ is even and
\[
\Arf(\phi)  = \begin{cases} 
	1	&{g \equiv 0,3 \pmod 4}\\
	0	&{g \equiv 1,2 \pmod 4}
\end{cases}
\] 
then label the curves as in Figure \ref{figure:curvelabels12}.

\item If $\gcd(\sing)$ is even and
\[
\Arf(\phi)  = \begin{cases} 
	1	&{g \equiv 1,2 \pmod 4}\\
	0	&{g \equiv 0,3 \pmod 4}
\end{cases}
\] 
then label the curves as in Figure \ref{figure:curvelabels3}.
\end{enumerate}
No matter the labeling scheme, define the curve system $\mathsf{C}(\sing, \Arf)$ to be the collection
\[\mathsf{C}(\sing ,\Arf ) = \{ a_i\}_{i=1}^{2g-1} \cup \left\{b_i: i = \sum_{j=1}^\ell k_j \text{ for } \ell=1, \ldots, n \right\},\]
where the $b_i$ indices are understood mod $2g-2$.
\end{definition}

\begin{figure}[ht]
\centering
\begin{subfigure}{\textwidth}
\centering
\labellist
\small
\pinlabel $a_0=b_0$ [l] at 112.8 20
\pinlabel $b_1$ [l] at 176 20
\pinlabel $b_2$ [l] at 240 20
\pinlabel $b_{g-3}$ [l] at 308 20
\pinlabel $b_{2g-8}$ [l] at 240 84.8
\pinlabel $b_{2g-7}$ [l] at 176 84.8
\pinlabel $b_{2g-6}$ [l] at 112.8 84.8
\pinlabel $b_{2g-5}$ [bl] at 48 95
\pinlabel $b_{2g-4}$ [t] at 52 60
\pinlabel $b_{2g-3}$ [l] at 40 6.4
\endlabellist
\includegraphics{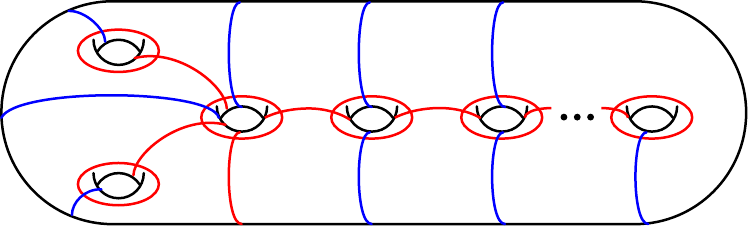}
\caption{Labels in cases (1) and (2) of Definition \ref{def:curvesystem}.}
\label{figure:curvelabels12}
\end{subfigure}\\
\begin{subfigure}{\textwidth}
\centering
\labellist
\small
\pinlabel $a_0=b_0$ [l] at 174.4 84.8
\pinlabel $b_1$ [l] at 108.8 84.8
\pinlabel $b_2$ [l] at 46.4 84.8
\pinlabel $b_3$ [l] at 108.8 21.6
\pinlabel $b_4$ [l] at 174.4 21.6
\pinlabel $b_5$ [l] at 238.4 21.6
\pinlabel $b_{g+1}$ [l] at 304 21.6
\endlabellist
\includegraphics{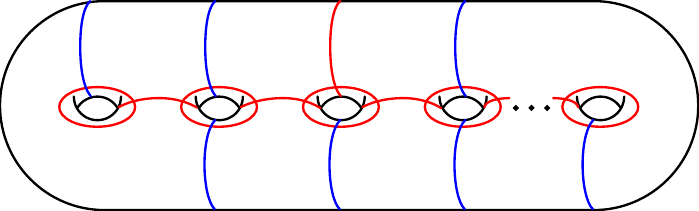}
\caption{Labels in case (3) of Definition \ref{def:curvesystem}.}
\label{figure:curvelabels3}
\end{subfigure}
\caption{Naming conventions for simple closed curves, depending on $\gcd(\sing)$, $\Arf$, and $g$.}
\label{figure:curvelabels}
\end{figure}

We first see that $\mathsf{C}(\sing, \Arf)$ determines a flat surface.

\begin{lemma}\label{lemma:flatprototype}
Suppose that $g \ge 3$. Let $\sing$ be any partition $\sing$ of $2g-2$, and if $\gcd(\sing)$ is even, choose $\Arf \in\{0,1\}$.
Then there exists a flat cone metric $\sigma$ on $\Sigma_g$ such that the curves of the curve system $\mathsf{C}(\sing, \Arf)$ are realized as cylinders on $(\Sigma, \sigma)$.
\end{lemma}
\begin{proof}
This is nothing more than an application of the Thurston--Veech construction; we recall the details for the interested reader below.

Suppose that $\gamma_h$ and $\gamma_v$ are any pair of multicurves which jointly fill $\Sigma_g$. We will eventually realize $\gamma_h (resp. \gamma_v)$ as a union of horizontal (resp. vertical) cylinders on a translation surface. Their union therefore defines a cellulation of $\Sigma_g$ whose 0--cells are the intersection points of $\gamma_h$ with $\gamma_v$, whose 1--cells are the simple arcs of $\gamma_h \cup \gamma_v$, and the 2--cells of which are $n$ polygonal disks with $2(m_1 + 2)$, $\ldots$, $2(m_n + 2)$ sides, respectively.

The dual complex is therefore built out of (topological) squares, with $2(m_i + 2)$ of them meeting around the $i^\text{th}$ vertex. Declaring each square to be a flat unit square yields a flat cone metric $\sigma$ on $\Sigma_g$ with cone angle
\[ \frac{\pi}{2} \cdot 2(m_i + 2) = m_i + 2\]
around the $i^\text{th}$ cone point $p_i$.

The intersection graph (c.f. Section \ref{subsection:generalslide}) associated to the curve system $\mathsf{C}(\sing, \Arf)$ is a tree, and therefore the curves can be partitioned into two multicurves $\gamma_h$ and $\gamma_v$. 
We call the flat cone metric $\sigma=\sigma(\sing, \Arf)$ resulting from the Thurston--Veech construction the {\em prototype} for the pair $(\sing, \Arf)$.
\end{proof}

\begin{remark}
Observe that by construction, $\sigma$ has cone points of angles
\[2(k_1 + 1) \pi, \ldots, 2(k_n + 1) \pi,\]
and if we assume that $\sigma$ does indeed come from an abelian differential $\omega$ (as established in Lemma \ref{lemma:prototype}), then the Arf invariant of the $r$--spin structure induced by $\omega$ agrees with the choice of $\Arf \in \{0,1\}$ used to construct it \cite[Lemma 5.4]{Calderon_strata}.
\end{remark}

\begin{lemma}\label{lemma:prototype}
Suppose that $g \ge 3$. Let $\sing$ be any partition $\sing$ of $2g-2$, and if $\gcd(\sing)$ is even, choose $\Arf \in\{0,1\}$ (if $g=3$, set $\Arf=1$).
Then there exists a non-hyperelliptic marked abelian differential $(X, f, \omega)$ in $\HT(\sing)$ such that the curves of the curve system $\mathsf{C}(\sing, \Arf)$ are realized as the vertical and horizontal cylinders of $(X, f, \omega)$.
\end{lemma}

Note that in the case $g=3$, the components of $\HM(4)$ and $\HM(2,2)$ with $\Arf=0$ coincide with the hyperelliptic components \cite[Theorem 2]{KZ_strata}.

\begin{proof} In order to prove that the metric $\sigma$ is induced from an abelian differential, we show that our prototype surface admits a translation vector field $V$ outside of the singularities. Therefore, by Lemma \ref{lemma:flatequiv}, there is some $(X, f, \omega) \in \HT(\sing)$ which is isometric to $(\Sigma_g, \sigma)$ via a marking $f: \Sigma_g \rightarrow X$. 

To build this vector field, we choose a positive horizontal direction on each square. If we can do this so that the squares glue consistently along each edge, we may then define $V$ by pasting together the constant horizontal vector fields $\langle 1, 0 \rangle$. The problem then becomes to find a coherent choice of positive horizontal direction for each square.

Observe that declaring the edges dual to the edges of $\gamma_v$ to be horizontal and the edges dual to $\gamma_h$ to be vertical naturally partitions the edges of the squares. Then the coherence condition on gluing squares is equivalent to the condition that the curves of $\gamma_h$ and $\gamma_v$ may be oriented so that each intersection of a curve of $\gamma_h$ with one of $\gamma_v$ is positively oriented. Now since the intersection graph of the multicurves $\gamma_h$ and $\gamma_v$ is a tree, one may choose the orientation of a single curve of $\gamma_h$ and extend by the positivity constraint to yield a coherent orientation on $\gamma_h$ and $\gamma_v$ (see Figure \ref{figure:orientation}). The choice of positive horizontal on each square induced from the orientation of $\gamma_h$ then yields the desired result.

\begin{figure}[ht]
\labellist
\large
\pinlabel \red{$\gamma_v$} at 70 15
\pinlabel \blue{$\gamma_h$} at 50 15
\endlabellist 
\includegraphics{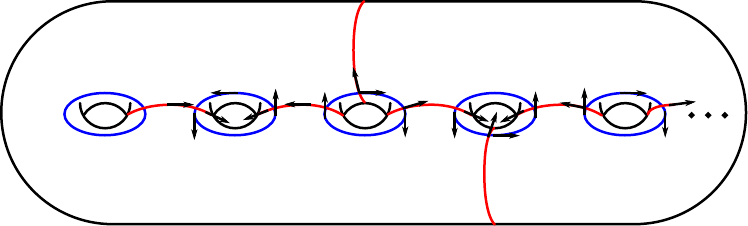}
\caption{Extending the orientation of a single curve to a global orientation of $\gamma_h \cup \gamma_v$.}
\label{figure:orientation}
\end{figure}

Suppose towards contradiction that $(X, \omega)$ is in a hyperelliptic component of $\HM(\sing)$; then by Theorem \ref{theorem:KZstrata}, we know that $\sing = (2g-2)$ or $(g-1, g-1)$. In this case, the hyperelliptic involution $\iota$ of $(X, \omega)$ must setwise fix each of its cylinders (see, e.g., \cite[Lemma 2.1] {Lind_hyp}). In particular, $\iota$ fixes the curves $\{a_i\}_{i=1}^{2g-1}$ but reverses their orientation. 

Therefore, each $a_i$ curve is the lift of an arc $\alpha_i$ on $X/ \iota$ connecting branch values of the associated cover $q: X \rightarrow X / \iota$. Now observe that (by the Birman--Hilden theory, see Section \ref{subsection:chainslide} or \cite{MargWin_BHSurvey}) the geometric intersection numbers of the $a_i$ are determined by the intersection numbers of the $\alpha_i$: indeed, one has
\[i(a_i, a_j) = 2 i(\alpha_i, \alpha_j) + e(\alpha_i, \alpha_j)\]
where $i(\alpha_i, \alpha_j)$ counts only intersection points in the interior of the arcs and $e(\alpha_i, \alpha_j) \in \{0, 1, 2\}$ is the number of their shared endpoints. But now since
\[i(a_5, a_4) = i(a_5, a_6) = i(a_5, a_0)=1\]
we know that $\alpha_5$ shares an endpoint with each of $\{\alpha_4, \alpha_6, \alpha_0\}$. However, since $\alpha_5$ has only two endpoints, this means that two of $\{\alpha_4, \alpha_6, \alpha_0\}$ share an endpoint, and hence their corresponding $a_i$ curves intersect, a contradiction.

Therefore $(X, \omega)$ cannot be hyperelliptic.
\end{proof}

We can now prove our main theorems. Since any twist in a cylinder of $(X, f, \omega)$ must stabilize the component of $\HT(\sing)$ in which it lies (Lemma \ref{lemma:cyltwist}), the monodromy theorem then follows from a quick application of Theorem \ref{theorem:genset}. The classification theorem is in turn an application of Proposition \ref{prop:monstab}.

\begin{proof}[Proof of Theorems \ref{theorem:classification} and \ref{theorem:moncomp}]
Let $g \ge 5$, let $\sing$ be a partition of $2g-2$, and choose ${ \mathcal H}$ to be a non-hyperelliptic connected component of $\HM(\sing)$.
Set $(X, f, \omega) \in \HT(\sing)$ to be the prototype for $(\sing, \Arf({ \mathcal H}))$ constructed in Lemma \ref{lemma:prototype}. Observe that $(X, \omega) \in { \mathcal H}$. Let $\phi$ be the $r$--spin structure induced by $(X, f, \omega)$, and define
\[\Gamma := \left \langle T_c : f(c) \text{ is a cylinder on }(X, \omega) \right \rangle.\]
By construction, the generating sets of cases 1 and 2 of Theorem \ref{theorem:genset} are realized as the core curves of cylinders on the appropriate prototype surface $(X, \omega)$; therefore
\[
\Mod_g[\tilde \phi] \le \Gamma,
\]
where $\tilde \phi$ is some $(2g-2)$-spin structure that refines $\phi$. 

It remains to show that $\Mod_g[\phi] \le \Gamma$. We begin by observing that for every $i$, the cut surface
\[ \Sigma_g \setminus ( b_0 \cup b_i \cup a_2\cup a_4 \cup \ldots\cup  a_{2g-2})\]
is the union of an $i+2$--holed and a $2g-i$--holed sphere, so homological coherence (Lemma \ref{lemma:homcoh}) implies that each curve $b_i$ has $\phi(b_i) = i$ (relative to an appropriate orientation). Therefore, by construction of the prototype $(X, f, \omega)$, we see that $\Gamma$ contains twists on curves $b_i$ with $\tilde \phi$--values
\[\left\{ \tilde \phi(b_i) \right\} = \left \{ \sum_{j=1}^\ell k_j : \ell=1, \ldots, n \right \}.\]
Since $r = \gcd(\sing)$, the set $\left\{ \tilde \phi(b_i) \right\}$ generates the subgroup $r \Z / (2g-2) \Z$ of $\Z/(2g-2) \Z$, and so Theorem \ref{theorem:genset}.\ref{case:r} implies that $\Gamma = \Mod_g[\phi].$

Putting this together with Lemmas \ref{lemma:cyltwist} and \ref{lemma:rspinconstant} yields
\begin{equation}\label{equation:monodromy}\Mod_g[\phi] = \Gamma \le \mathcal{G}({ \mathcal H}) \le \Mod_g[\phi]\end{equation}
and therefore all of the groups are equal, completing the proof of Theorem \ref{theorem:moncomp}.

Finally, Proposition \ref{prop:monstab} together with Lemma \ref{lemma:stabindex} imply that there are exactly
\[\left[ \Mod(\Sigma_g) : \Mod_g[\phi] \right]
= \left\{ \begin{array}{ll}
r^{2g} & \text{ if } r \text { is odd} \\
(r/2)^{2g} \left( 2^{g-1} (2^g + 1) \right) & \text{ if } r \text { is even and} \Arf({ \mathcal H}) = 0 \\
(r/2)^{2g} \left( 2^{g-1} (2^g - 1) \right) & \text{ if } r \text { is even and} \Arf({ \mathcal H}) = 1
\end{array}\right.\]
components of $\HT(\sing)$ lying over ${ \mathcal H}$. 
Combining the above statements for the components of $\HM(\sing)$, as classified by Theorem \ref{theorem:KZstrata}, completes the proof of Theorem \ref{theorem:classification}.
\end{proof}

Now that we have characterized the geometric monodromy group of $\mathcal H$, we can quickly recover the homological monodromy.

\begin{proof}[Proof of Corollary \ref{corollary:GR}]
The monodromy $\overline \rho: \pi_1^{\text{orb}}({ \mathcal H}, (X, \omega)) \to \Sp(2g;\Z)$ of the bundle $H_1({ \mathcal H})$ factors through the geometric monodromy $\rho: \pi_1^{\text{orb}}({ \mathcal H}, (X, \omega)) \to \Mod_g[\phi]$: 
\[
\overline \rho = \Psi \circ \rho.
\]
By Theorem \ref{theorem:moncomp}, $\rho$ surjects onto the spin structure stabilizer $\Mod_g[\phi]$. The result now follows from the description of the action of the $r$-spin mapping class group on homology (Lemma \ref{step1}).
\end{proof}

Finally, using our description of which deformations can occur in a stratum, we can also give a description of which curves appear as cylinders on a surface in a stratum.

\begin{proof}[Proof of Corollary \ref{corollary:cylchar}]
We first consider the case when $\tilde { \mathcal H}$ is a non-hyperelliptic component of $\HT(\sing)$. Let $\phi$ denote the corresponding $r$--spin structure. Recall that we are trying to prove that a curve $c$
is realized as the core curve of a cylinder on some marked abelian differential in $\tilde { \mathcal H}$ if and only if it is nonseparating and $\phi$--admissible.

Lemma \ref{lemma:cylinderadmissible} shows that the core curve of every cylinder on every $(X, f, \omega) \in \tilde { \mathcal H}$ is $\phi$--admissible, and by Stokes' theorem, no separating curve can ever be a cylinder on an abelian differential.

To see that the conditions are also sufficient, let $(X, f, \omega)$ be any marked abelian differential in $\tilde { \mathcal H}$ (for example, the prototype coming from Lemma \ref{lemma:prototype}) and let $\xi$ be a cylinder on $(X, \omega)$. Suppose that the core curve of $\xi$ is $f(d)$, where $d$ is a simple closed curve on $\Sigma_g$. By Lemma \ref{lemma:cylinderadmissible}, $d$ is $\phi$--admissible. 

As explained in Lemma \ref{lemma:ccpcurves}, the spin stabilizer subgroup $\Mod_g[\phi]$ acts transitively on the set of admissible curves, and hence there is some $g \in \Mod_g[\phi]$ so that $g(d)=c$. Therefore, $f \circ g^{-1}(c) = f(d)$ is the core curve of $\xi$ on $(X, \omega)$, and hence $c$ is realized as the core curve of a cylinder on
\[g \cdot (X, f, \omega) = (X, f g^{-1}, \omega).\]
By Proposition \ref{prop:monstab}, we have that $(X, fg^{-1}, \omega)$ is in $\tilde { \mathcal H}$, finishing the proof.

Suppose now that $\tilde { \mathcal H}$ is a hyperelliptic component of $\HT(\sing)$ with corresponding hyperelliptic involution $\iota$; then as in the proof of Lemma \ref{lemma:prototype}, $\sing = (2g-2)$ or $(g-1, g-1)$ and $\iota$ fixes the core curves of each cylinder. Therefore, the core curve of each cylinder on any $(X, f, \omega) \in \tilde { \mathcal H}$ is the lift of a simple arc of $X / \iota$.

To see that every nonseparating curve fixed by $\iota$ is the core curve of a cylinder, let $c$ be such a curve. As in the previous case, pick some $(X, f, \omega) \in \tilde { \mathcal H}$ and a cylinder on it with core curve $f(d)$. Let $\gamma$ and $\delta$ denote the (simple) arcs of $\Sigma_g / \iota$  corresponding to $c$ and $d$, which connect the branch values of the associated cover $q: X \rightarrow X /\iota$. 

We now recall that the hyperelliptic component ${ \mathcal H} \subset \HM(\sing)$ is an orbifold $K(\pi, 1)$ for (an extension of) a surface braid group on $X /\iota$ \cite[\S1.4]{LM_strata}. In particular, its geometric monodromy group $\mathcal{G}({ \mathcal H})$ contains a copy of the entire braid group $B_q$ on the set of branch values of $q$ which lift to regular points of $(X, \omega)$ (compare \cite[Proof of Corollary 2.6]{Calderon_strata}). Since such a braid group acts transitively on the set of simple arcs connecting its points, we know there is an element of $B_q$ taking $\delta$ to $\gamma$; hence by the Birman--Hilden theory (see Section \ref{subsection:chainslide} or \cite{MargWin_BHSurvey}) there is an element $g \in \mathcal{G}({ \mathcal H})$ taking $d$ to $c$.

As above, the curve $c$ is the core curve of a cylinder on
\[g \cdot (X, f, \omega) = (X, f g^{-1}, \omega),\]
and by Proposition \ref{prop:monstab}, we have that $(X, fg^{-1}, \omega)$ is in $\tilde { \mathcal H}$, finishing the proof.
\end{proof}

\end{document}